\documentclass[11pt,A4paper]{article}
\usepackage[dvips]{graphicx}
\usepackage{bm}
\usepackage{amsmath,amsthm,amssymb,citesort}
\newtheorem{thm}{Theorem}[section]
\newtheorem{theorem}[thm]{Theorem}
\newtheorem{proposition}[thm]{Proposition}
\newtheorem{lemma}[thm]{Lemma}
\newtheorem{definition}[thm]{Definition}

\newtheorem{claim}[thm]{Claim}
\newtheorem{corollary}[thm]{Corollary}

\newcommand{\bp}{{\bm p}}

\newcommand{\bmm}{{\bm m}}

\newcommand{\bd}{{\bm d}}

\newcommand{\calM}{{\cal M}}
\newcommand{\calN}{{\cal N}}
\newcommand{\bR}{{\bm R}}
\newcommand{\bS}{{\bm S}}
\newcommand{\Bvector}[2]{\stackrel{#1}{\mathstrut #2}}

\begin{document}
\title{Rooted-tree Decompositions with Matroid Constraints and the Infinitesimal Rigidity of Frameworks with Boundaries}
\author{Naoki Katoh\footnote{Department of Architecture and Architectural Engineering, Kyoto University {\tt naoki@archi.kyoto-u.ac.jp}} 
\and 
Shin-ichi Tanigawa\footnote{Research Institute for Mathematical Sciences, Kyoto University  {\tt tanigawa@kurims.kyoto-u.ac.jp}}
}
\maketitle
\begin{abstract}
As an extension of a classical tree-partition problem,
we consider  decompositions of graphs into edge-disjoint (rooted-)trees with an additional matroid constraint.
Specifically, suppose we are given a graph $G=(V,E)$, a multiset $\bR=\{r_1,\dots, r_t\}$ of vertices in $V$, and a matroid ${\cal M}$ on $\bR$.
We prove a necessary and sufficient condition for $G$ to be decomposed into 
$t$ edge-disjoint subgraphs $G_1=(V_1,T_1), \dots, G_t=(V_t,T_t)$ such that 
(i) for each $i$, $G_i$ is a tree with $r_i\in V_i$, and 
(ii)  for each $v\in V$, the multiset $\{r_i\in \bR\mid v\in V_i\}$ is a base of ${\cal M}$.
If ${\cal M}$ is a free matroid, this is a decomposition into $t$ edge-disjoint spanning trees;
thus, our result is a proper extension of Nash-Williams' tree-partition theorem.

Such a matroid constraint is motivated by combinatorial rigidity theory.
As a direct application of our decomposition theorem,  we  present characterizations of the infinitesimal rigidity of frameworks with non-generic ``boundary'',
which  extend classical Laman's theorem for generic 2-rigidity of bar-joint frameworks and Tay's theorem for generic $d$-rigidity of body-bar frameworks.
\end{abstract}

%
%
%

\section{Introduction}
In this paper two fundamental results in combinatorial optimization, Tutte-Nash-Williams tree-packing theorem and Nash-Williams tree-partition theorem, are extended.
In 1961 Tutte~\cite{Tutte:1961} and Nash-Williams~\cite{Nash:1961} independently proved that
an undirected graph $G=(V,E)$ contains $k$ edge-disjoint spanning trees if and only if
$|\delta_G({\cal P})|\geq k|{\cal P}|-k$ holds for any partition ${\cal P}$ of $V$, where $\delta_G({\cal P})$ denotes
the set of edges of $G$ connecting  two distinct subsets of ${\cal P}$ 
and $|{\cal P}|$ denotes the number of subsets of ${\cal P}$.
As a dual form, Nash-Williams tree-partition theorem~\cite{nash1964decomposition} asserts that
an undirected graph $G=(V,E)$ can be decomposed into $k$ edge-disjoint spanning trees if and only if 
$|E|=k|V|-k$ and $|F|\leq k|V(F)|-k$ for any non-empty $F\subseteq E$, where $V(F)$ denotes the set of vertices incident to $F$.

These two theorems are sometimes referred to in terms of rooted-edge-connectivity,
as  edge-disjoint spanning trees indicate how to send distinct ``commodities'' from a specific root-node to other vertices without interference.
(In fact, the packing of spanning trees is an equivalent concept to rooted-edge-connectivity, see e.g.,~\cite{Frank2011}.)  
In this paper we address a more general situation. 
Suppose we have $t$ distinct roots, each of which has an ability of sending a commodity, and suppose 
the set of commodities possesses an independence structure, say, linear independence by regarding commodities as vectors. 
Then we are asked to decide whether one can send commodities from roots to every vertex  so that 
each vertex receives $k$ independent commodities without transmitting more than two distinct commodities through an edge.
This paper provides a polynomial time algorithm to answer to this question.

The study is  motivated by combinatorial rigidity theory.
One of major topics in rigidity theory is to describe a rigidity condition of architectural frameworks in terms of the underlying graphs,
where the connection to tree-packing condition (and its variants) has been particularly investigated in the literature (see e.g.,\cite{whiteley:88,tay:84,white:whiteley:87}). 
Based on this background together with our new decomposition theorem, we obtain extensions of two fundamental theorems in combinatorial rigidity theory,  
Laman's theorem for generic 2-rigidity of bar-joint frameworks and Tay's theorem for generic $d$-rigidity of body-bar frameworks.

\subsection{Rooted-tree Decompositions}
For a graph $G=(V,E)$, a pair $(T,r)$ of $T\subseteq E$ and $r\in V$ is called a {\em rooted-tree} if either (i) $T=\emptyset$ or 
(ii) $T$ is connected without cycles and $r\in V(T)$. Here $r$ is called a {\em root} of $T$.
For a rooted-tree $(T,r)$, we denote the {\it set} $V(T)\cup \{r\}$ by $V(T,r)$, 
and we say that $v\in V$ is {\em spanned}  by $(T,r)$ if $v\in V(T,r)$. 
Note that $V(T,r)=V(T)$ if $T\neq\emptyset$; otherwise $V(T,r)=\{r\}$ (which is not equal to $V(T)=\emptyset$).

As we mentioned, our focus is on a decomposition of a graph into edge-disjoint rooted-trees of specific roots.
For simplicity, a pair $(G, \bR)$ of a graph $G$ and a multiset $\bR$ of vertices (that specify {\em roots}) is called a {\em graph with roots}.
\begin{definition}
\normalfont
Let $(G,\bR)$ be a graph with roots $\bR=\{r_1,r_2,\dots, r_t\}$ and  ${\cal M}$ be a matroid on $\bR$.
Rooted-trees $(T_1,r_1), \dots, (T_t,r_t)$ are called {\em edge-disjoint} if $T_i\cap T_j=\emptyset$ for $1\leq i<j\leq t$;
they are said to be {\em basic} if the multiset $\{r_i\in \bR\mid  v\in V(T_i,r_i)\}$ is a base of ${\cal M}$ for each $v\in V$. 
We say that $(G,\bR)$ admits {\em a basic rooted-tree decomposition with respect to $\calM$} 
(or simply, a basic decomposition) if 
the edge set can be partitioned into basic edge-disjoint rooted-trees $(T_1,r_1), \dots, (T_t,r_t)$, (where $T_i=\emptyset$ is allowed).
\end{definition}

Figure~\ref{fig:example} shows an example for the case when $\calM$ is a graphic matroid.

For each $v\in V$ and $F\subseteq E$,  let $\bR_v=\{r_i\in \bR\mid r_i=v\}$
and  $\bR_F=\{r_i\in \bR\mid r_i\in V(F)\}$ as multi-subsets of $\bR$.
The following main theorem characterizes the decomposability into basic edge-disjoint rooted-trees.
\begin{theorem}
\label{thm:partition}
Let $G=(V,E)$ be a graph with a multiset $\bR=\{r_1,\dots, r_t\}$ of vertices, and ${\cal M}$ be a matroid on $\bR$ of rank $k$ and the rank function $r_{\cal M}:2^{\bR}\rightarrow \mathbb{Z}$.
Then, $(G,\bR)$ admits a basic rooted-tree decomposition with respect to $\calM$ if and only if
$(G,\bR)$ satisfies the following three conditions:
\begin{description} 
\item[(C1)] $\bR_v$ is independent in ${\cal M}$ for each $v\in V$;
\item[(C2)] $|F|+|\bR_F|\leq k|V(F)|-k+r_{\cal M}(\bR_F)$ for any non-empty $F\subseteq E$;
\item[(C3)] $|E|+|\bR|=k|V|$.
\end{description}
\end{theorem}
Notice that, if ${\cal M}$ is a free matroid, this coincides with Nash-Williams' tree-partition theorem.
In Theorem~\ref{thm:dual}, we give a dual form of Theorem~\ref{thm:partition} as a proper extension of Tutte-Nash-Williams' tree-packing theorem.

Throughout the paper, we will refer to the conditions given in Theorem~\ref{thm:partition} as (C1), (C2) and (C3) with respect to $\calM$, respectively. 
Checking (C2) can be easily reduced to a submodular function minimization and thus done in polynomial time.
In Section~\ref{sec:alg} we present an efficient algorithm via matroid intersection. 

Note that, even though checking (C2) can be reduced to matroid intersection, 
this fact alone does not imply Theorem~\ref{thm:partition}.
Indeed, if ${\cal M}$ can be written as the direct sum of $k$ matroids of rank $1$, Theorem~\ref{thm:partition} straightforwardly follows from the matroid union theorem;
however for general ${\cal M}$ Theorem~\ref{thm:partition} has no clear (and direct) connection to the matroid union theorem.

\begin{figure}[th]
\centering
\begin{minipage}{0.4\textwidth}
\centering
\includegraphics[scale=1]{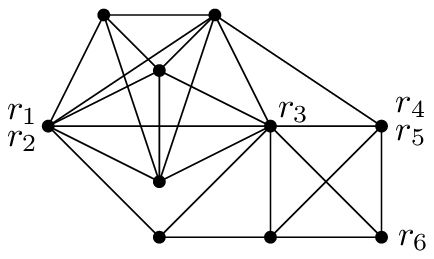}
\par
(a)
\end{minipage}
\begin{minipage}{0.4\textwidth}
\centering
\includegraphics[scale=1]{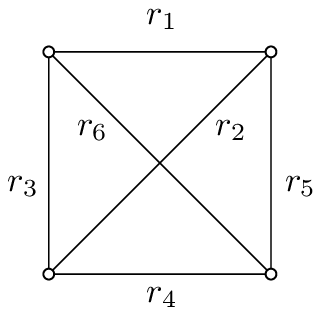}
\par
(b)
\end{minipage}

\vspace{\baselineskip}

\begin{minipage}{0.98\textwidth}
\centering
\begin{minipage}{0.32\textwidth}
\centering
\includegraphics[scale=1]{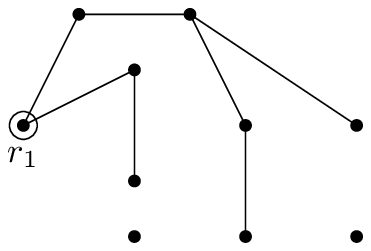}
\par 
$(T_1,r_1)$
\end{minipage}
\begin{minipage}{0.32\textwidth}
\centering
\includegraphics[scale=1]{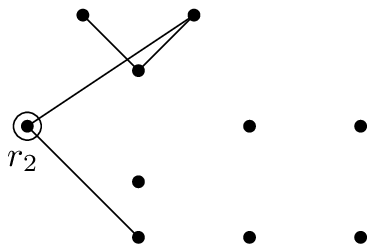}
\par 
$(T_2,r_2)$
\end{minipage}
\begin{minipage}{0.32\textwidth}
\centering
\includegraphics[scale=1]{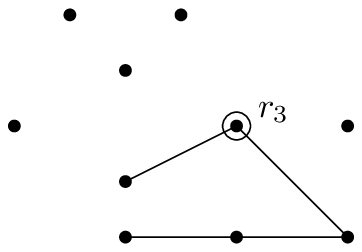}
\par
$(T_3,r_3)$
\end{minipage}

\vspace{\baselineskip}

\begin{minipage}{0.32\textwidth}
\centering
\includegraphics[scale=1]{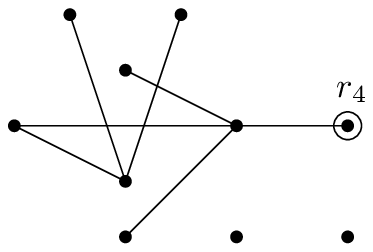}
\par
$(T_4,r_4)$
\end{minipage}
\begin{minipage}{0.32\textwidth}
\centering
\includegraphics[scale=1]{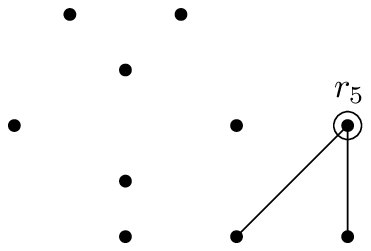}
\par
$(T_5,r_5)$
\end{minipage}
\begin{minipage}{0.32\textwidth}
\centering
\includegraphics[scale=1]{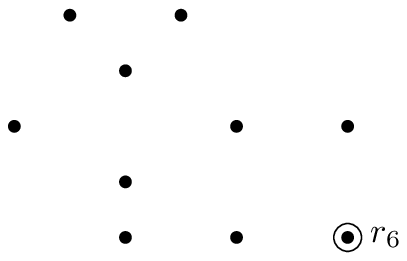}
\par
$(T_6,r_6)$
\end{minipage}

\vspace{0.5\baselineskip}
(c)
\end{minipage}
\caption{(a) A graph $G$ with roots $\bR=\{r_1,\dots, r_6\}$. (b) A graph representing a graphic matroid $\calM$ on $\bR$. 
(c) A basic rooted-tree decomposition. Each vertex is spanned by exactly three rooted-trees whose roots form a spanning tree in the graph (b).}
\label{fig:example}
\end{figure}

\subsection{Related Works}
Nash-Williams' tree-partition theorem is nowadays a special case of the matroid union theorem, 
as it is equivalent to packing bases of the graphic matroid of $G$  (see e.g.,\cite{Schriver,Frank2011}).
For applications to rigidity theory, Whiteley~\cite{whiteley:88} discussed a generalization of Nash-Williams' theorem by mixing spanning trees and spanning pseudoforests.
(A graph is said to be a spanning pseudoforest if each connected component contains exactly one cycle).
Based on the matroid union theorem, he observed  
that, for two integers $k$ and $l$ with $k\geq l$, $G=(V,E)$ can be partitioned into $l$ edge-disjoint spanning trees and $k-l$ spanning pseudoforests
if and only if $|E|=k|V|-l$ and $|F|\leq k|V(F)|-l$ for any non-empty $F\subseteq E$.
The range of $l$ was later broadened by  Haas~\cite{haas:2002}.
Algorithms for checking these counting conditions or computing decompositions were discussed 
in e.g.~\cite{Imai:1983,sugihara1985detection,gabow:1992,berg:jordan:2003b,lee:streinu:2005,ito2011constant}.

These types of matroids are referred to as {\em count matroids}~\cite{Frank2011} or {\em sparsity matroids}, and have a wide range of applications in combinatorial geometry, including rigidity theory (see, e.g.,\cite{Whitley:1997}).
Our primary motivation of this study is indeed to extend the decomposition theory of these count matroids to more general forms.
For this purpose,  we have presented a special case of Theorem~\ref{thm:partition} in \cite{rooted_forest} where ${\cal M}$ is restricted to a variant of uniform matroid.

Another direction of related research is the packing of branchings into digraphs.
A directed forest, called a {\em branching}, is a digraph in which the in-degree of each node is at most one.
The set of nodes of in-degree $0$ is called the root-set.
For $R\subseteq V$, a branching is said to be a {\em spanning} branching with roots $R$ if every vertex can reach to a root in $R$.
The well-known Edmonds branching-theorem~\cite{edmonds1972edge} is a good characterization of  a digraph $D=(V,A)$ with a given collection of root-sets 
$\{R_1,\dots, R_k\}$ to contain $k$ arc-disjoint spanning branchings with roots $R_i$.
However, Edmonds' branching theorem can produce only {\em spanning} branchings;
in general, the problem of answering whether there exist $k$ arc-disjoint branchings spanning a proper subset of $V$ is known to be NP-complete, 
and only a few special cases are known to be solvable in polynomial time~\cite{berczi2009packing,kamiyama2009arc,fujishige-note}.

Even in the undirected case, the problem becomes intractable if we drop the term ``spanning'' from the decomposition.
In fact, the problem of deciding whether an undirected graph can be partitioned into two edge-disjoint trees is known to be NP-complete~\cite{Palvolgyi}. 
Our main theorem (Theorem~\ref{thm:partition}) however asserts that one can actually relax the condition of ``spanning'' by introducing an appropriate matroid constraint.

\subsection{Applications to Rigidity Theory}
Theorem~\ref{thm:partition} has various applications to rigidity theory.
A {\em bar-joint framework} is a structure consisting of bars connected by universal joints at endpoints as shown in Figure~\ref{fig:bar_joint}(a). 
The underlying graph is obtained by associating each joint with a vertex and each bar with an edge,
thus  a bar-joint framework can be identified with a pair $(G,{\bm p})$ of a graph $G$ and ${\bm p}:V\rightarrow \mathbb{R}^d$.
Celebrated Laman's theorem~\cite{laman:Rigidity:1970} asserts that
$(G,\bp)$ is minimally rigid on a generic $\bp$ in the plane if and only if $|E|=2|V|-3$ and $|F|\leq 2|V(F)|-3$ for any nonempty $F\subseteq E$,
where $\bp$ is called {\em generic} if the set of coordinates is algebraically independent over $\mathbb{Q}$.
See, e.g.,~\cite{graver:servatius:servatius:CombinatorialRigidity:1993} for formal definition.

Although characterizing generic 3-dimensional rigidity of bar-joint frameworks is recognized as one of the most difficult open problems in this field, 
there are solvable structural models even in higher dimension.
One of the fundamental results in this direction is a combinatorial characterization of generic rigidity of {\em body-bar frameworks} shown by Tay~\cite{tay:84}.
Body-bar frameworks consist of disjoint rigid bodies articulated by bars as illustrated in Figure~\ref{fig:body_bar}(a), and the underlying graphs are extracted 
by associating each body with a vertex and each bar with an edge.
Tay~\cite{tay:84} proved that the generic rigidity of {\em body-bar frameworks} can be characterized in terms of the underlying graphs by Nash-Williams' condition for decomposing into ${d+1\choose 2}$ 
spanning trees.

In this paper, replacing Nash-Williams' theorem with Theorem~\ref{thm:partition}, we obtain extensions of Laman's theorem and Tay's theorem to the models with {\em boundary}.
In most applications, especially in engineering context, a framework has a relation to the external environment, where
several joints/bodies are connected to the ground or walls. 
Figure~\ref{fig:bar_joint}(b) and Figure~\ref{fig:body_bar}(b)(c) show typical examples:
Figure~\ref{fig:bar_joint}(b) illustrates a so-called  {\em pinned bar-joint framework}, where three joints are fixed in the space;
in Figure~\ref{fig:body_bar}(b) and (c) illustrate body-bar counterparts, where several bodies are linked to the ground by bars or pins.
This motivates us to investigate {\em frameworks with boundary}. 

Frameworks with boundary are indeed an old concept even in the mathematical study of rigidity (see~\cite{jordan2010pinned} for survey and fundamental facts).
In fact, combinatorial characterizations of these models straightforwardly follow from Laman's theorem or Tay's theorem, if we assume ``genericity'' of configuration of boundary.
For example, to extend Laman's theorem to pinned bar-joint frameworks,  we just need to observe that
a 2-dimensional pinned bar-joint framework is rigid if and only if 
there are at least two pinned joints and connecting all pairs of pinned joints results in a rigid framework (without pinning).
This fact combined with Laman's theorem implies a combinatorial characterization of 2-dimensional pinned bar-joint frameworks for generic rigidity.
This straightforward extension  however requires that $\bp$ should be generic and in particular pinned joints have to be generic, 
which cannot be achieved in most applications as joints are usually pinned down on the ground or walls.   

Motivated by these practical requirements, we shall address the problem of coping with ``non-generic'' boundaries.
Our new results assert that, even without genericity assumption for boundary condition, a naturally extended statement is true for characterizing infinitesimal rigidity.
Although the formal description will be given in Sections~\ref{sec:body_bar} and \ref{sec:bar_joint},
counting conditions (C1)(C2)(C3) of Theorem~\ref{thm:partition} will naturally appear as a necessary condition for the infinitesimal rigidity of frameworks with ``non-generic'' boundary,
and the existence of basic rooted-tree decompositions enables us to  show even the sufficiency.

Below, we list structural models we address in this paper:
\begin{itemize}
\item bar-joint frameworks with {\em bar-boundary} in $\mathbb{R}^2$, in which the Pl{\"u}cker coordinate of each boundary-bar is predetermined (Theorem~\ref{thm:bar_joint_bar});
\item bar-joint frameworks with {\em pin-boundary} in $\mathbb{R}^2$, in which the coordinate of each pin is predetermined (Theorem~\ref{thm:pin_laman_projective});
\item bar-joint frameworks with {\em slider-boundary} in $\mathbb{R}^2$, in which the direction of each slider is predetermined (Theorem~\ref{thm:bar_slider});
\item body-bar frameworks with {\em bar-boundary} in $\mathbb{R}^d$, in which the Pl{\"u}cker coordinate of each boundary-bar is predetermined (Theorem~\ref{thm:body_bar_bar});
\item body-bar frameworks with {\em pin-boundary} in $\mathbb{R}^d$, in which the coordinate of each pin is predetermined (Theorem~\ref{thm:body_bar}).
\end{itemize}
The second one (Theorem~\ref{thm:pin_laman_projective}) was recently observed by Servatius, Shai and Whiteley~\cite{servatius2010combinatorial} for engineering applications,
where the proof is done by the the so-called Henneberg construction. We shall present it as a corollary of a more general statement (Theorem~\ref{thm:bar_joint_bar}).
We should note that main results of \cite{servatius2010combinatorial,servatius2010geometric} are a combinatorial characterization of {\em assur graphs} and their geometric properties in the plane.
Our new observations for body-bar frameworks might be useful for developing a higher dimensional counterpart. 

2-dimensional bar-joint frameworks with slider-boundary (called {\em bar-joint-slider frameworks}) were previously studied in Streinu and Theran~\cite{streinu2010slider},
where an interesting relation between decompositions and non-generic realizations was observed.
Theorem~\ref{thm:bar_slider}, which is  a corollary of Theorem~\ref{thm:pin_laman_projective},  extends their result.
(This result was already presented in a conference~\cite{isaac2009} without detailed proof.)

\begin{figure}[t]
\centering
\begin{minipage}{0.4\textwidth}
\centering
\includegraphics[width=0.5\textwidth]{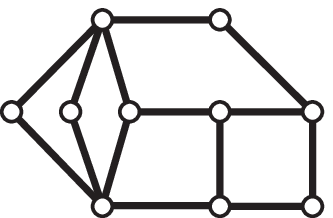}
\par
(a)
\end{minipage}
\begin{minipage}{0.4\textwidth}
\centering
\includegraphics[width=0.5\textwidth]{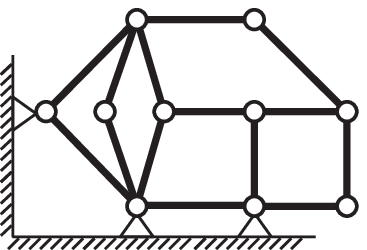}
\par
(b)
\end{minipage}
\caption{(a) Bar-joint framework. (b) Pinned bar-joint framework.}
\label{fig:bar_joint}
\end{figure}

\begin{figure}[t]
\centering
\begin{minipage}{0.24\textwidth}
\centering
\includegraphics[scale=0.48]{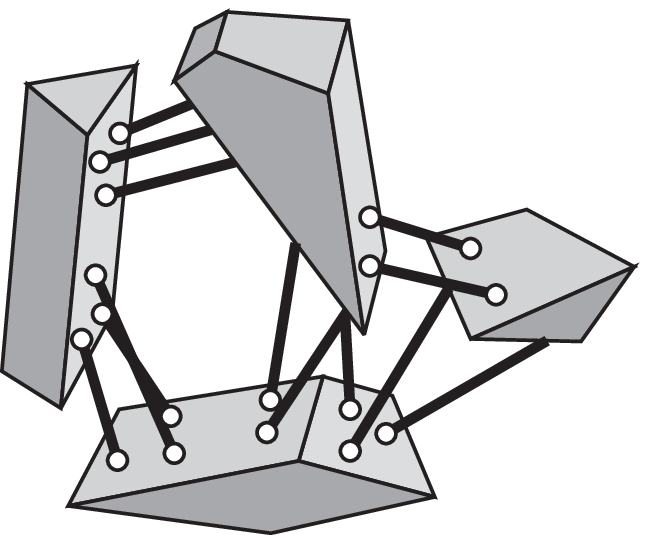}
\par
(a)
\end{minipage}
\begin{minipage}{0.37\textwidth}
\centering
\includegraphics[scale=0.45]{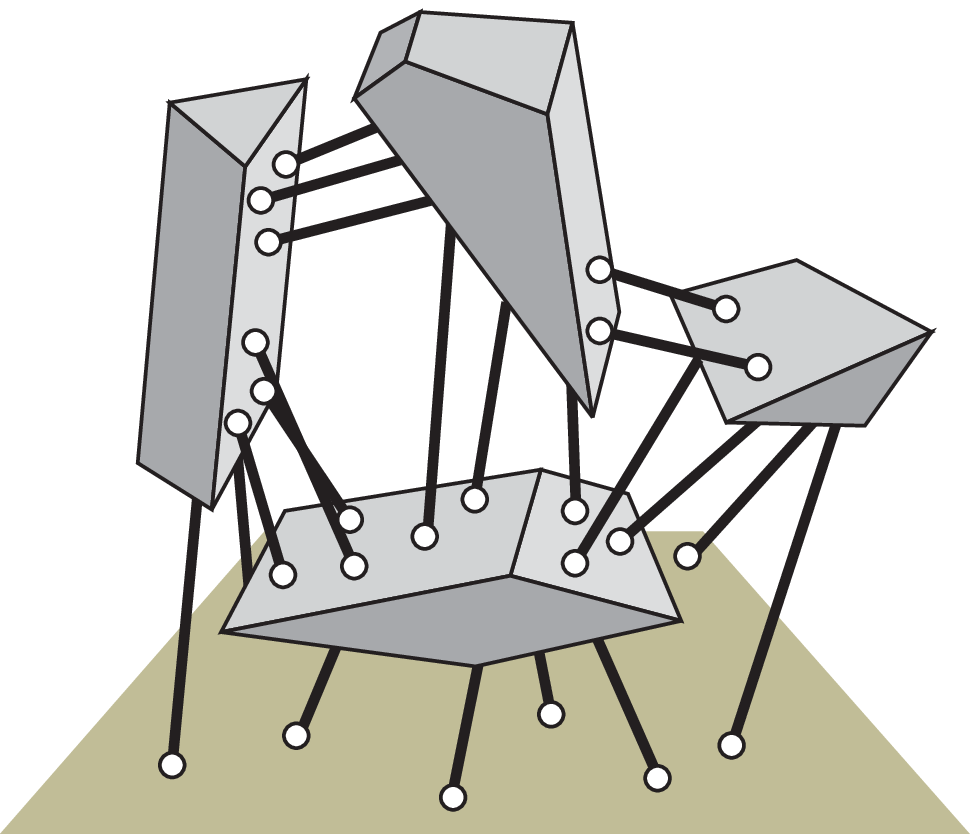}
\par
(b)
\end{minipage}
\begin{minipage}{0.37\textwidth}
\centering
\includegraphics[scale=0.46]{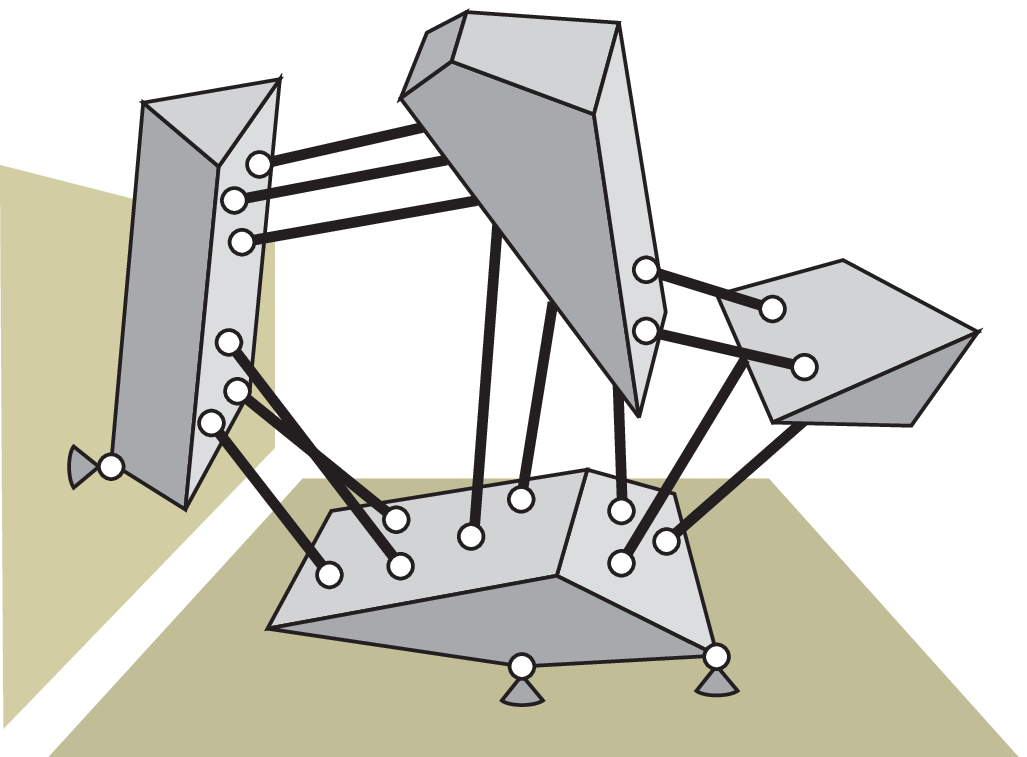}
\par
(c)
\end{minipage}
\caption{(a)Body-bar framework. (b)Body-bar framework with bar-boundary. (c)Body-bar framework with pin-boundary (pinned body-bar framework).}
\label{fig:body_bar}
\end{figure}

\subsection{Organizations}
We first review a combinatorial background in Section~\ref{sec:preliminaries}
and then present a proof of Theorem~\ref{thm:partition} in Section~\ref{sec:partition}.
In Section~\ref{sec:alg} we discuss computational issues.
In Section~\ref{sec:dual} we present a dual form of Theorem~\ref{thm:partition}.
Applications of basic-decompositions to rigidity theory are discussed in Sections~\ref{sec:body_bar} and~\ref{sec:bar_joint}.
We conclude the paper by listing remarks.

\section{Preliminaries}
\label{sec:preliminaries}
For a matroid $\calM=(S,{\cal I})$ on a finite set $S$, the rank function of $\calM$ is denoted by $r_{\calM}:S\rightarrow \mathbb{Z}$.
$r_{\calM}(S)$ is especially called the {\em rank of $\calM$}, which is simply denoted by $r_{\calM}$.
A set $X\subseteq S$ is called a {\em spanning set} of ${\cal M}$ if $r_{\cal M}(X)=r_{\calM}$. 
For $X\subseteq S$, let ${\rm sp}_{\calM}(X)=\{x\in S\mid r_{\calM}(X+x)=r_{\calM}(X)\}$.
The {\em restriction} of $\calM$ to $X\subseteq S$ is ${\cal M}|X=(X,\{I\in{\cal I}\mid I\subseteq X\})$, which forms a matroid on $X$.
The {\em truncation} of ${\cal M}$ is defined as the one of rank function $r^{\downarrow}(X)=\min\{r_{\calM}(X),r_{\cal M}(S)-1 \} \ (X\subseteq S)$. 
An element $x\in X$ is called a {\em coloop} if $r_{\calM}(S-x)<r_{\calM}(S)$.
$x\in X$ is said to be {\em parallel} to $y\in X$ if $r_{\calM}(\{x,y\})=r_{\calM}(\{x\})=r_{\calM}(\{y\})=1$.

We will use the following preliminary result 
concerning the  matroid induced by a monotone submodular function, which can be found in e.g.~\cite[Chapter~12]{oxley}.
The function $f:2^S\to \mathbb{R}$ is called {\em submodular} if $f(X)+f(Y)\geq f(X\cup Y)+f(X\cap Y)$ for any $X,Y\subseteq S$
and {\em monotone} if $f(X)\leq f(Y)$ for any $X\subseteq Y\subseteq S$.
Also $f$ is called {\em intersecting submodular} if the submodular inequality holds for every pair $X,Y\subset S$ with $X\cap Y\neq \emptyset$.

Let $f:2^S\to \mathbb{Z}$ be an integer-valued monotone submodular function.
It is known that $f$ {\em induces} a matroid on $S$, denoted by ${\cal N}(f)$, 
whose collection of independent sets is written by
${\cal I}=\{I\subseteq S\mid  |I'|\leq f(I') \text{ for all non-empty } I'\subseteq I\}$.
The following proposition provides an explicit formula expressing the rank function $r_{\calN(f)}$ of ${\cal N}(f)$, see e.g.,~\cite{Schriver,fujishige,Frank2011}.
\begin{proposition}
\label{prop:rank}
Let $f$ be an integer-valued monotone submodular function on $S$ 
satisfying $f(X)\geq 0$ for every non-empty $X\subseteq S$.
Then, for any non-empty $X\subseteq S$, the rank of $X$ in ${\cal N}(f)$ is given by
\begin{equation}
\label{eq:rank1}
\mbox{$r_{\calN(f)}(X)=\min\{|X_0|+\sum_{i=1}^m f(X_i)\}$},
\end{equation}
where the minimum is taken over all partitions $\{X_0,X_1,\dots, X_m\}$ of $X$ such that $X_i\neq \emptyset$ for each $i=1,\dots,m$ (and $X_0$ may be empty).
\end{proposition}
%

\section{Proof of Theorem~\ref{thm:partition}} 
\label{sec:partition}
Let $G=(V,E)$ be a graph with roots $\bR=\{r_1,\dots, r_t\}$, ${\cal M}=(\bR,{\cal I})$ be a matroid on $\bR$ with rank $k$ and the rank function $r_{\cal M}$.
We begin with an easier direction, the necessity of Theorem~\ref{thm:partition}.
\begin{proof}[Proof of the necessity of Theorem~\ref{thm:partition}]
For a basic decomposition, (C1) is obviously necessary.

To see (C2) and (C3), let us take a basic rooted-tree decomposition $\{(T_1,r_1),\dots, (T_t,r_t)\}$ of $(G,\bR)$ with respect to ${\cal M}$, where $t=|\bR|$.
$(T_i,r_i)$ can be converted to an arborescence (i.e., a directed tree) 
by assigning an orientation so that each vertex in $V(T_i,r_i)\setminus \{r_i\}$ has exactly one entering arc (and $r_i$ has no entering arc).
Since the decomposition is basic,  the sum of $|\bR_v|$ and the number of edges entering to $v$ is equal to $k$ for each $v\in V$.
This implies $|\bR|+|E|=k|V|$, and thus (C3) holds.

To see (C2),  let us consider $F\subseteq E$, and let $K\subseteq E$ be the set of edges oriented from a vertex in $V\setminus V(F)$ to a vertex in $V(F)$.
For the same reason as above, we have $|\bR_F|+|F|+|K|=k|V(F)|$.
Moreover, since the decomposition is basic, $|K|+r_{\cal M}(\bR_F)\geq k$ holds.
These imply $|\bR_F|+|F|\leq k|V(F)|-k+r_{\cal M}(\bR_F)$.
\end{proof}

For an integer  $c$, 
we define a set function $f_{{\cal M},c}:2^E\rightarrow \mathbb{Z}$ by 
\begin{equation}
f_{{\cal M},c}(F)=c(|V(F)|-1)-(|\bR_F|-r_{\cal M}(\bR_F)) \quad (F\subseteq E). 
\end{equation}
\begin{lemma}
\label{lem:sub}
Let $(G,\bR)$ be a graph with roots, $\calM$ be a matroid on $\bR$, and $c$ be an integer. 
Suppose (C1) is satisfied and $c\geq r_{\calM}$.
Then, $f_{\calM,c}$ is an integer-valued  monotone submodular function.
\end{lemma}
\begin{proof}
It is known that, for any ${\bm b}:2^V\rightarrow \mathbb{Z}_+$, 
the set function $g:2^E\rightarrow \mathbb{Z}_+$ defined by $g(F)=\sum_{v\in V(F)}{\bm b}(v) \ (F\subseteq E)$ is monotone and submodular (see e.g.,~\cite{Frank2011}). 
We now have $f_{\calM,c}(F)=\sum_{v\in V(F)} (c-|\bR_v|)-c+r_{\calM}(\bR_F)$.
Since $c-|\bR_v|\geq 0$ by (C1) and $c\geq r_{\calM}$, $f_{\calM,c}(F)$ is monotone and submodular. 
\end{proof}
Thus, if (C1) is satisfied and $c\geq r_{\calM}$, $f_{\calM,c}$ induces a matroid on $E$, which is denoted by $\calN(f_{\calM,c})$.
Note that $(G,\bR)$ satisfies (C2) with respect to ${\cal M}$ if and only if $E$ is independent in $\calN(f_{\calM,k})$. 

To show the sufficiency, we begin with an easy observation.
$(G,\bR)$ is called {\em disconnected} if $G$ is not connected.
A {\em connected component} $(G',\bR')$ is a subgraph of $(G,\bR)$, where $G'=(V',E')$ is a connected component of $G$ and $\bR'=\bR_{E'}$. 
\begin{lemma}
\label{lem:connected}
Let $(G,\bR)$ be a disconnected graph with roots, and $\calM$ be a matroid on $\bR$ of rank $k$.
Suppose (C1), (C2) and (C3) are satisfied.
Then, for each connected component $(G',\bR')$ of $(G,\bR)$, 
$\bR'$ is a spanning set of ${\cal M}$, and $(G',\bR')$ satisfies (C1), (C2) and (C3) with respect to $\calM|\bR'$.
\end{lemma}
\begin{proof}
Let $(G'=(V',E'),\bR')$ be a connected component. Clearly, $(G',\bR')$ satisfies (C1).

From (C2) of $(G,\bR)$, we have $|E'|\leq f_{\calM,k}(E')$ and $|E\setminus E'|\leq f_{\calM,k}(E\setminus E')$.
From (C3), $k|V|=|E|+|\bR|$. Also, $r_{\calM}(\bR_F)\leq k$ for $F\subseteq E$ since $k$ is the rank of ${\cal M}$.
Combining these relations, 
we have $k|V|=|E|+|\bR|=|E\setminus E'|+|E'|+|\bR\setminus \bR'|+|\bR'|\leq f_{\calM,k}(E\setminus E')+f_{\calM,k}(E')+|\bR\setminus \bR'|+|\bR'|=k|V|-2k+r_{\calM}(\bR_{E\setminus E'})+r_{\calM}(\bR_{E'})\leq k|V|$.
In other words, the equality holds in each inequality, and in particular we have  $|E'|=f_{\calM,k}(E')$ and $r_{\calM}(\bR')=k$.
This implies the first part of the claim.
Note $|E'|=f_{\calM,k}(E')=k|V(E')|-k-|\bR_{E'}|+r_{\calM}(\bR')=k|V'|-|\bR'|$. This implies (C3) of $(G',\bR')$.
Also, for any $F\subseteq E'$, we have $|F|\leq f_{\calM,k}(F)=f_{\calM\mid \bR',k}(F)$, implying (C2) of $(G',\bR')$.
\end{proof}

Let us move to the proof of the sufficiency of our main theorem.
\begin{proof}[Proof of the sufficiency of Theorem~\ref{thm:partition}]
The proof is done by induction on $|E|$.
Note that, if $E=\emptyset$, Theorem~\ref{thm:partition} trivially follows from (C1) and (C3), and hence we shall consider the case $|E|>0$.
If $G$ is disconnected, we can consider each connected component separately by Lemma~\ref{lem:connected}.
We thus assume that $G$ is connected.

For $F\subseteq E$, $(G[F],\bR_F)$ denotes the subgraph {\em edge-induced} by $F$. Namely, $G[F]=(V(F),F)$. 
A non-empty $F\subseteq E$ is said to be {\em tight} if $|F|=f_{\calM,k}(F)$.
A tight set $F$ is called {\em proper} if $V(F)\neq V$.
We begin with investigating properties of proper tight sets.
\begin{claim}
\label{claim:consequence}
Suppose $G$ has a proper tight set $F$.
Let $s=r_{\calM}(\bR_F)$.
Then there is an $F'\subseteq F$ satisfying the following two properties: 
\begin{description}
\item[(i)] $(G[F]-F',\bR_F)$ admits a basic rooted-tree decomposition with respect to ${\calM}|\bR_F$,
\item[(ii)] $F'$ can be partitioned into  $k-s$ edge-disjoint spanning trees on $V(F)$.
\end{description}
\end{claim}
Figure~\ref{fig:F} shows an example for a proper tight set $F$ in the graph illustrated in Figure~\ref{fig:example}.
\begin{proof}
Take a vertex $v'\in V(F)$.
By (C1), we have $|\bR_{v'}|=r_{\calM}(\bR_{v'})\leq s$.
We insert $(k-s)$ copies of $v'$ into $\bR$ as new roots (if $s<k$), and let $\bR'$ be the resulting multiset.
A new matroid ${\cal M'}$ on $\bR'$ is constructed based on $\calM$ by adding these copies as coloops.
Namely, $r_{\calM'}(\bR'_F)=r_{\calM}(\bR_F)+(k-s)=k$.
We now show 
\begin{equation}
\label{claim:consequence1}
\text{$(G[F],\bR'_F)$ satisfies (C1)(C2)(C3) with respect to $\calM'|\bR'_F$}.
\end{equation}
Clearly (C1) is satisfied.
Since each element of $\bR'\setminus \bR$ is inserted as a coloop in ${\cal M}'$, 
we have $f_{{\calM,k}}=f_{{\calM',k}}$ and thus ${\cal N}(f_{\calM,k})={\cal N}(f_{\calM',k})$.
By the independence of $F$ in ${\cal N}(f_{\calM,k})$, $F$ is also independent in ${\cal N}(f_{\calM',k})$, implying (C2).
Furthermore, since $|F|=f_{\calM,k}(F)=k|V(F)|-k-|\bR_F|+s=k|V(F)|-|\bR'_F|$,
(C3) is satisfied.

Thus, by induction on the size of edge set, $(G[F],\bR'_F)$ admits a basic rooted-tree decomposition $\{(T_1,r_1),\dots, (T_{t'},r_{t'})\}$, where $\bR'_F=\{r_1,\dots, r_{t'}\}$.
Without loss of generality, let $(T_1,r_1),\dots, (T_{k-s},r_{k-s})$ be the rooted-trees among them whose roots belong to $\bR'\setminus \bR$.
Since $\{r_i\in \bR'_F \mid v\in V(T_i,r_i)\}$ is a base of $\calM'|\bR'_F$ for each $v\in V(F)$, every vertex $v$ of $V(F)$ must be spanned by $(T_i,r_i)$ for all $i=1,\dots,k-s$.
Let $F'=\bigcup_{i=1}^{k-s} T_i$. Then $F'$ has the desired property.
\end{proof}

\begin{figure}[t]
\centering
\begin{minipage}{0.5\textwidth}
\centering
\includegraphics[scale=1]{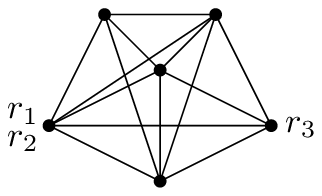}
\par
(a)
\end{minipage}

\begin{minipage}{0.98\textwidth}
\centering
\begin{minipage}{0.23\textwidth}
\centering
\includegraphics[scale=1]{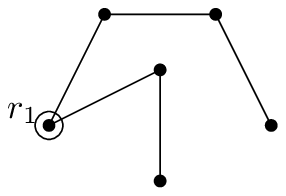}
\par 
$(F_1,r_1)$
\end{minipage}
\begin{minipage}{0.23\textwidth}
\centering
\includegraphics[scale=1]{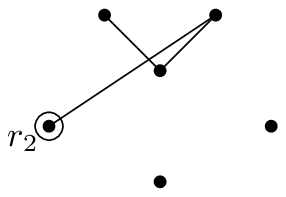}
\par 
$(F_2,r_2)$
\end{minipage}
\begin{minipage}{0.23\textwidth}
\centering
\includegraphics[scale=1]{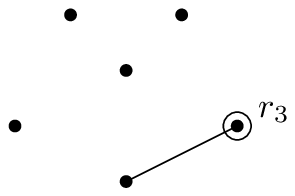}
\par
$(F_3,r_3)$
\end{minipage}
\begin{minipage}{0.23\textwidth}
\centering
\includegraphics[scale=1]{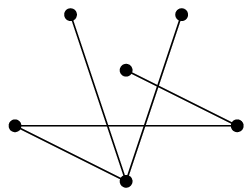}
\par
$F'$
\end{minipage}

\vspace{0.5\baselineskip}
(b)
\end{minipage}
\caption{(a) An unbalanced proper tight set $F$ of $(G,\bR)$ shown in Figure~\ref{fig:example}, 
where $\bR_F=\{r_1,r_2,r_3\}$ and $s=r_{\calM}(\bR_F)=2$.
(b) A spanning tree $F'$ on $V(F)$ and a basic rooted-tree decomposition of $(G[F]-F',\bR_F)$.}
\label{fig:F}
\end{figure}

A tight set $F$ is called {\em unbalanced} if there is a vertex $v\in V(F)$ satisfying ${\rm sp}_{\calM}(\bR_v)\neq {\rm sp}_{\calM}(\bR_F)$; Otherwise $F$ is called {\em balanced}.
A proper tight set given in Figure~\ref{fig:F} is an example of unbalanced one. 
We now consider the case where $(G,\bR)$ has an unbalanced proper tight set.
\begin{claim}
\label{claim:tight}
Suppose $G$ has an unbalanced proper tight set $F$.
Then, $(G,\bR)$ admits a basic rooted-tree decomposition.
\end{claim}
\begin{proof}
Let $t'=|\bR_F|$ and $s=r_{\calM}(\bR_F)$. Without loss of generality, we denote $\bR_F=\{r_1,\dots, r_{t'}\}\subset \bR=\{r_1,\dots, r_{t'},r_{t'+1},\dots, r_t\}$.
By Claim~\ref{claim:consequence}, $F$ can be  partitioned into $\{F_1,\dots, F_{t'}, F'\}$
such that $(F_1,r_1),\dots, (F_{t'},r_{t'})$ are  basic edge-disjoint rooted-trees with respect to $\calM|\bR_F$ and  $F'$ is the union of edge-disjoint $(k-s)$ spanning trees on $V(F)$.
(See Figure~\ref{fig:F} for an example.)
Then, we have
\begin{equation}
\label{eq:card1}
|F\setminus F'|=f_{{\cal M}|\bR_F,s}(F\setminus F')=s|V(F)|-|\bR_F|.
\end{equation}
Note also
\begin{equation}
\label{eq:F_F'}
F\setminus F'\neq \emptyset;
\end{equation}
otherwise, $F_i=\emptyset$ for all $i=1,\dots, t'$; as the decomposition $\{(F_1,r_1),\dots, (F_{t'},r_{t'})\}$ is basic with respect to ${\cal M}|\bR_F$,
we have $|\bR_v|=r_{\cal M}(\bR_F)=s$; thus, $F$ becomes balanced, a contradiction.
%

Based on $\{F_1,\dots, F_{t'}, F'\}$, we now construct a new graph $(G'=(V,E'),\bR')$ with roots in the following way:
\begin{itemize}
\item Remove $F\setminus F'$ from $G$ and remove $\bR_F$ from $\bR$;
\item For each $i=1,\dots, t'$ and for each $v\in V(F_i,r_i)$, insert a copy of $v$ into $\bR\setminus \bR_F$ as a new root.
This copy is denoted by $r_v^i$.
\end{itemize}
In total  we inserted $s$ copies of each $v\in V(F)$ into $\bR\setminus \bR_F$ as new roots,
since there are exactly $s$ rooted-trees among $\{(F_1,r_1),\dots, (F_{t'},r_{t'})\}$ that span $v\in V(F)$.
An example is given in Figure~\ref{fig:G'}(a). 
We denote the multiset of these new roots by ${\bm S}$ 
(i.e., ${\bm S}=\{r_v^i \mid v\in V(F_i,r_i), 1\leq i\leq t'\}$).
We have thus constructed a new graph $(G'=(V,E'),\bR')$ with $E'=E\setminus (F\setminus F')$ and $\bR'=\bR\setminus \bR_F\cup {\bm S}$. 
From (\ref{eq:card1}) and the construction, we have
\begin{equation}
\label{eq:S}
|\bR_F|+|F\setminus F'|=s|V(F)|=|\bS|.
\end{equation}

A new matroid ${\calM}'$ on $\bR'$ is constructed from ${\calM}$ as follows.
 For each $i$ with $1\leq i\leq t'$ and for each $v\in V(F_i,r_i)$, we insert $r_v^i$ into ${\calM}$ so that $r_v^i$ is parallel to $r_{i}\in \bR$ (in the sense of matroids).
 We then obtained a matroid ${\calM}^*$ on the multiset $\bR\cup {\bm S}$.
After removing all elements of $\bR_F$, a matroid, $\calM'=\calM^*\setminus \bR_F$, on $\bR'$ is defined.
(See Figure~\ref{fig:G'}(b).)
From the construction we have,  for each $v\in V(F)$,
\begin{equation}
\label{eq:suf1_span1}
{\rm sp}_{\calM^*}(\bR'_v)={\rm sp}_{\calM^*}({\bm S}_v)={\rm sp}_{\calM^*}(\bR_F).
\end{equation}
We now claim the following:
\begin{equation}
\label{eq:claim:1}
\text{$(G',\bR')$ satisfies (C1) (C2) (C3) with respect to $\calM'$.}
\end{equation}
Assuming (\ref{eq:claim:1}) for a while, let us show how to construct a basic decomposition of $(G,\bR)$.
By (\ref{eq:F_F'}), $|E'|<|E|$ holds, and hence we can apply the inductive hypothesis to $(G',\bR')$.
Namely, $(G',\bR')$ admits a basic rooted-tree decomposition by induction (see Figure~\ref{fig:G'}(c)).
Recall that $\bR'$ consists of $\bR\setminus \bR_F=\{r_{t'+1},\dots, r_{t}\}$ and 
$\bS=\{r_v^i\mid v\in V(F_i,r_i),  1\leq i\leq t'\}$.
It is thus convenient to denote the corresponding rooted-trees of the decomposition by $(T_{t'+1},r_{t'+1}),\dots, (T_{t},r_t)$ and  $\{(T_v^i,r_v^i)\mid v\in V(F_i,r_i),1\leq i\leq t'\}$.
(So $\{T_{t'+1},\dots, T_t\}\cup \{T_v^i\mid v\in V(F_i,r_i),1\leq i \leq t'\}$ is a partition of $E'$ 
into edge-disjoint trees.) 
Note that, for any $u,v\in V(F)$ and any $r_v^i\in \bR'_v$, $r_v^i\in {\rm sp}_{\calM^*}(\bR'_u)$ holds by (\ref{eq:suf1_span1}).
This implies  that $(T_v^i,r_v^i)$ cannot span $u$ from the basicness; in other words,
\begin{equation}
\label{eq:suf21}
V(T_v^i,r_v^i)\cap V(F)=\{v\} \qquad \text{ for every $r_v^i\in \bS$}.
\end{equation}
 
We are now ready to construct a basic rooted-tree decomposition $\{(T_1^*, r_1),\dots, (T_{t}^*,r_{t})\}$ of $(G,\bR)$ with respect to $\calM$.
For each $1\leq i\leq t$, we define $T^*_i$ by
\begin{equation}
\label{eq:def_T^*}
T^*_i=
\begin{cases}
F_i\cup \left(\bigcup_{v\in V(F_i,r_i)} T_v^i \right) & \text{ for }1\leq i\leq t' \\
T_i & \text{ for } t'+1\leq i\leq t.
\end{cases}
\end{equation}
Clearly, $T^*_i$ is connected with $r_i\in V(T_i^*)$ (if $T_i^*\neq \emptyset$).
By (\ref{eq:suf21}), $T^*_i$ has no cycle,  and thus $(T_i^*,r_i)$ is a rooted-tree.
Also, it is not difficult to see that each vertex $v$ is spanned by $k$ rooted-trees since there are exactly $k$ indices ``$i$'' for which $T_v^i$ or $T_i$ span $v$.
We now check that this decomposition is indeed basic.

Consider $v\in V$, and suppose $v\in V(T_u^i,r_u^i)$ for some $r_u^i\in\bS$.
From the construction of $\bS$ there is $r_i\in \bR_F$ such that $u\in V(F_i,r_i)$;
hence  we obtain $v\in V(T_i^*,r_i)$ from definition (\ref{eq:def_T^*}).
Namely, $v\in V(T_u^i,r_u^i)$ implies $v\in V(T_i^*,r_i)$.
Since $r_u^i$ is parallel to $r_i$ in ${\cal M}^*$, this implies
\begin{align*}
{\rm sp}_{\calM^*}(\{r_u^i\in \bS \mid v\in V(T_u^i,r_u^i)\})\subseteq {\rm sp}_{\calM^*}(\{r_i\in \bR_F \mid v\in V(T_i^*,r_i)\}.
\end{align*}
for each $v\in V$. 
Also, for each $v\in V$,
\begin{align*}
\{r_i\in \bR\setminus \bR_F \mid v\in V(T_i^*,r_i)\} &=
\{r_i\in \bR\setminus \bR_F \mid v\in V(T_i,r_i)\}, 
\end{align*}
from definition (\ref{eq:def_T^*}).
We thus obtain, for each $v\in V$, 
\begin{equation}
\label{eq:suf22}
\begin{split}
&r_{\calM^*}(\{r_i\in \bR \mid v\in V(T_i^*,r_i)\})  \\
&\geq r_{\calM^*}(\{r_i\in \bR\setminus \bR_F \mid v\in V(T_i,r_i)\}\cup \{r_u^i\in \bS \mid v\in V(T_u^i,r_u^i)\})=k.
\end{split}
\end{equation}
Since  each vertex is spanned by $k$ rooted-trees among $\{(T_i^*,r_i)\mid r_i\in \bR\}$,
(\ref{eq:suf22}) implies that  $\{(T_i^*,r_i)\mid  r_i\in \bR \}$ is basic.
We thus obtained a basic rooted-tree decomposition of $(G,\bR)$.

\begin{figure}[t]
\centering
\begin{minipage}{0.4\textwidth}
\centering
\includegraphics[scale=1]{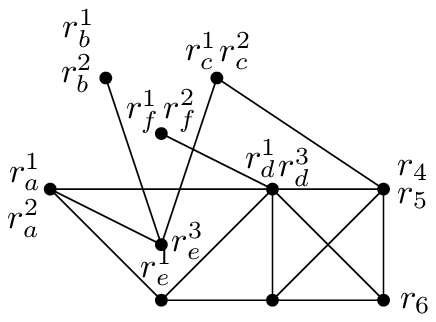}
\par
(a)
\end{minipage}
\begin{minipage}{0.4\textwidth}
\centering
\includegraphics[scale=1]{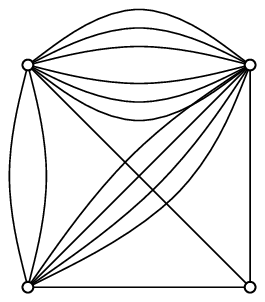}
\par
(b)
\end{minipage}

\vspace{\baselineskip}

\begin{minipage}{0.98\textwidth}
\centering
\begin{minipage}{0.32\textwidth}
\centering
\includegraphics[scale=1]{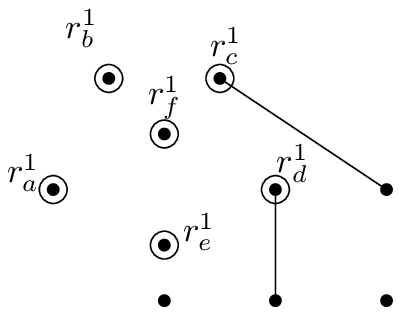}
\par 
$\{(T_v^1,r_v^1)\mid v\in V(F_1,r_1)\}$
\end{minipage}
\begin{minipage}{0.32\textwidth}
\centering
\includegraphics[scale=1]{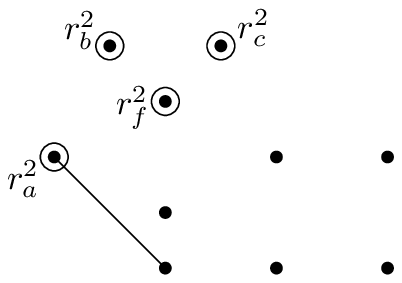}
\par 
$\{(T_v^2,r_v^2)\mid v\in V(F_2,r_2)\}$
\end{minipage}
\begin{minipage}{0.32\textwidth}
\centering
\includegraphics[scale=1]{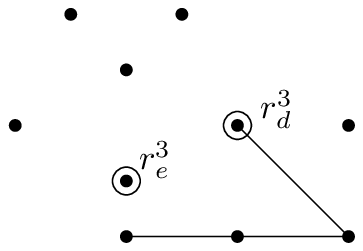}
\par
$\{(T_v^3,r_v^3)\mid v\in V(F_3,r_3)\}$
\end{minipage}

\vspace{\baselineskip}

\begin{minipage}{0.32\textwidth}
\centering
\includegraphics[scale=1]{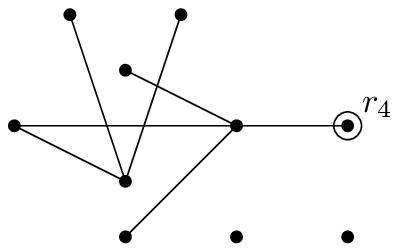}
\par
$(T_4,r_4)$
\end{minipage}
\begin{minipage}{0.32\textwidth}
\centering
\includegraphics[scale=1]{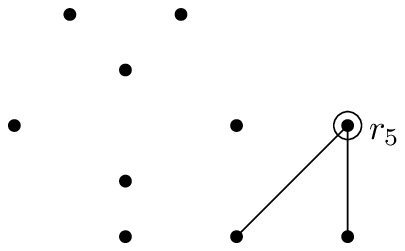}
\par
$(T_5,r_5)$
\end{minipage}
\begin{minipage}{0.32\textwidth}
\centering
\includegraphics[scale=1]{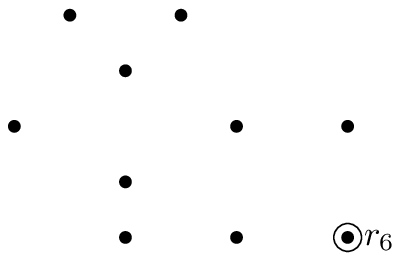}
\par
$(T_6,r_6)$
\end{minipage}

\vspace{0.5\baselineskip}
(c)
\end{minipage}
\caption{(a) $(G',\bR')$ obtained from $(G,\bR)$ given in Figure~\ref{fig:example}. (b) A graph representing $\calM'$. 
 (c) A basic rooted-tree decomposition of $(G',\bR')$.}
\label{fig:G'}
\end{figure}

The remaining thing is thus to prove (\ref{eq:claim:1}).
Clearly, (C1) is satisfied.
To see (C3), note $k=r_{\calM'}$ by (\ref{eq:suf1_span1}).
Also, by using (\ref{eq:S}), we obtain $|\bR'|=|\bR|-|\bR_F|+|\bS|=|\bR|+|F\setminus F'|$.
This yields $|E'|+|\bR'|=(|E|-|F\setminus F'|)+(|\bR|+|F\setminus F'|)=|E|+|\bR|=k|V|$, implying (C3).

To see (C2), suppose for a contradiction that there is $C$ with $C\subseteq E'$ that violates (C2).
Namely, $|C|\geq f_{{\calM'},k}(C)+1$.
Since $C$ satisfies $|C|\leq f_{{\calM},k}(C)$, we must have
\begin{equation}
\label{eq:suf11}
V(C)\cap V(F)\neq\emptyset.
\end{equation}
Also, since $C\subseteq E'=E\setminus (F\setminus F')$ and $F'$ is the union of edge-disjoint $(k-s)$ spanning trees with $F'\subseteq F$,
$C\cap F$ can be partitioned into $(k-s)$ edge-disjoint forests, which implies
\begin{equation}
\label{eq:suf12}
|C\cap F|\leq (k-s)(|V(C\cap F)|-1)
\end{equation}
if $C\cap F\neq \emptyset$.

By  (\ref{eq:suf11}) we have $\bS_C\neq \emptyset$, and 
hence ${\rm sp}_{\calM^*}({\bm S}_{C})= {\rm sp}_{\calM^*}(\bR_F)$ by (\ref{eq:suf1_span1}). 
As $\bR'=(\bR\setminus \bR_F)\cup \bS$, this yields 
${\rm sp}_{\calM^*}(\bR'_{C})={\rm sp}_{\calM^*}((\bR_C\setminus \bR_F)\cup \bS_C)={\rm sp}_{\calM^*}(\bR_C\cup \bR_F)={\rm sp}_{\calM^*}(\bR_{C\cup F})$.
Therefore, by $r_{{\cal M}'}(\bR'_C)=r_{\calM^*}(\bR'_C)$ and $r_{\calM}(\bR_{C\cup F})=r_{\calM^*}(\bR_{C\cup F})$, we obtain
\begin{equation}
\label{eq:suf13}
r_{\calM'}(\bR'_{C})=r_{\calM}(\bR_{C\cup F}).
\end{equation}
We also need one more relation:
\begin{equation}
\label{eq:suf14}
|\bR_{C\cup F}|+s|V(C)\cap V(F)|=|\bR'_{C}|+|\bR_F|,
\end{equation}
which can be obtained as follows:
\begin{align*}
|\bR_{C\cup F}|=\sum_{v\in V(C\cup F)}|\bR_v|
&=\sum_{v\in V(C)\setminus V(F)} |\bR_v|+\sum_{v\in V(F)} |\bR_v| \\
&=\sum_{v\in V(C)}|\bR'_v|-\sum_{v\in V(C)\cap V(F)} |\bR'_v|+|\bR_F| \\
&=|\bR'_{C}|-s|V(C)\cap V(F)|+|\bR_F|,
\end{align*}
where we used $\bR_v=\bR'_v$ for $v\in V(C)\setminus V(F)$ according to the definition of $\bR'$.
In total, if $C\cap F\neq \emptyset$,
\begin{align*}
&|C\cup F|=|C|+|F|-|C\cap F| \\
&\geq f_{\calM',k}(C)+1+f_{\calM,k}(F)-(k-s)(|V(C\cap F)|-1) \qquad (\text{by (\ref{eq:suf12})}) \\
&=k|V(C\cup F)|+k|V(C)\cap V(F)|-2k \\ 
& \hspace{7em} -|\bR'_{C}|-|\bR_F|+r_{\calM'}(\bR'_{C})+s-(k-s)(|V(C\cap F)|-1)+1 \\
&= k|V(C\cup F)|+k|V(C)\cap V(F)|-k-|\bR'_{C}|-|\bR_F|+r_{\calM'}(\bR'_{C})-(k-s)|V(C\cap F)|+1 \\
&\geq k|V(C\cup F)|+s|V(C)\cap V(F)|-k-|\bR'_{C}|-|\bR_F|+r_{\calM'}(\bR'_{C})+1 \\ 
& \hspace{23em} \text{ (by  $V(C\cap F)\subseteq V(C)\cap V(F)$)} \\
&=k|V(C\cup F)|-k-|\bR_{C\cup F}|+r_{\calM'}(\bR'_{C})+1 \qquad \text{ (by (\ref{eq:suf14}))} \\
&=k|V(C\cup F)|-k-|\bR_{C\cup F}|+r_{\calM}(\bR_{C\cup F})+1 \qquad \text{ (by (\ref{eq:suf13}))} \\
&=f_{\calM,k}(C\cup F)+1.
\end{align*}
On the other hand, if $C\cap F=\emptyset$,
\begin{align*}
|C\cup F|&\geq f_{\calM',k}(C)+1+f_{\calM,k}(F) \\
&= k|V(C\cup F)|+k|V(C)\cap V(F)|-2k-|\bR'_{C}|-|\bR_F|+r_{\calM'}(\bR'_{C})+s+1 \\
&= k|V(C\cup F)|+(k-s)|V(C)\cap V(F)|-2k-|\bR_{C\cup F}|+r_{\calM'}(\bR'_{C})+s+1 \quad \text{ (by (\ref{eq:suf14}))} \\
&\geq k|V(C\cup F)|-k-|\bR_{C\cup F}|+r_{\calM'}(\bR'_{C})+1 \qquad \text{ (by (\ref{eq:suf11}) and $k\geq s$)} \\
&= k|V(C\cup F)|-k-|\bR_{C\cup F}|+r_{\calM}(\bR_{C\cup F})+1 \qquad \text{ (by (\ref{eq:suf13}))} \\
&=f_{\calM,k}(C\cup F)+1.
\end{align*}
In either case $|C\cup F|>f_{\calM,k}(C\cup F)$. 
Since $C\cup F$ is an edge subset of $G$, this contradicts that  $G$ satisfies (C2) with respect to $\calM$. Thus (\ref{eq:claim:1}) is verified,
and the proof of Claim~\ref{claim:tight} is completed.
\end{proof}

By Claim~\ref{claim:tight}, we now consider the case where $(G,\bR)$ has no unbalanced proper tight set in the subsequent discussion.
Note that, in this situation, we have $r_{\calM}(\bR_F)<k$  for any proper tight set $F$ 
since any proper tight set $F$ with $r_{\calM}(\bR_F)=k$ cannot be balanced. 
(If  $F$ is balanced with $r_{\calM}(\bR_F)=k$, then $|\bR_v|=r_{\calM}(\bR_v)=r_{\calM}(\bR_F)=k$ for each $v\in V(F)$, and we will have $|F|=f_{\calM,k}(F)=k|V(F)|-|\bR_F|=0$.)

We say that an edge $uv\in E$ is {\em good} if 
${\rm sp}_{\cal M}(\bR_u)\neq {\rm sp}_{\calM}(\bR_v)$.
The following is the final claim.
\begin{claim}
\label{claim:suf_final}
There is a good edge in $G$.
\end{claim}
\begin{proof}
Suppose every vertex $v\in V$ satisfies $r_{\calM}(\bR_v)=k$.
By (C1), $|\bR_v|=k$. This implies $E=\emptyset$ by (C2), a contradiction.

Thus  there is a vertex $u$ with $r_{\calM}(\bR_u)<k$.
Suppose there is no good edge in $G$.
Every $uv\in E$ incident to $u$ satisfies ${\rm sp}_{\calM}(\bR_u)={\rm sp}_{\calM}(\bR_v)$ 
since otherwise $uv$ becomes good.
Since $G$ is connected, we consequently have ${\rm sp}_{\calM}(\bR_u)={\rm sp}_{\calM}(\bR_v)$ for every $v\in V$ by applying the same argument to the neighbors.
This implies ${\rm sp}_{\calM}(\bR_u)={\rm sp}_{\calM}(\bR)$ and hence $r_{\calM}<k$, a contradiction.
\end{proof}

We are now ready to construct a basic rooted-tree decomposition of $(G,\bR)$.
Let $uv\in E$ be a good edge shown in Claim~\ref{claim:suf_final}.
Since ${\rm sp}_{\calM}(\bR_u)\neq {\rm sp}_{\calM}(\bR_v)$, without loss of generality, we can assume ${\rm sp}(\bR_v)\not\subseteq {\rm sp}(\bR_u)$.
Then there is an $r_{j}\in \bR_v$ such that $\bR_u\cup \{r_{j}\}$ is still independent in ${\cal M}$.
Let us prepare a copy $r$ of $u$ as a new root and let $\bR''=\bR\cup \{r\}$ be a new multiset.
A new matroid ${\calM}''$ on $\bR''$ is constructed from $\calM$ by inserting $r$ as a parallel element to $r_j$.
Also, let $G''=(V,E'')$ be the graph obtained from $G$ by removing $uv$.
We claim the following:
\begin{equation}
\label{eq:claim:G''}
\text{$(G'',\bR'')$ satisfies (C1)(C2)(C3) with respect to ${\calM}''$.}
\end{equation}

Clearly, $(G'',\bR'')$ satisfies (C1), as $\bR_u\cup \{r\}$ is independent.
Also, since $|E''|+|\bR''|=|E|+|\bR|$, (C3) is also satisfied.
What remains is to show (C2).
Note $|E''|=f_{\calM'',k}(E'')$ by $|E''|=|E|-1$ and $f_{\calM,k}(E)=f_{\calM'',k}(E'')+1$.
This implies that (C2) is satisfied for any $F\subseteq E''$ with $V(F)=V$.
Suppose $(G'',\bR'')$ does not satisfy (C2).
Then, there is a $C\subset E''$ such that $|C|>f_{\calM'',k}(C)$, $u\in V(C)$, and $V(C)\neq V$.
Combining the following three inequalities, $f_{\calM,k}(C)-1\leq f_{\calM'',k}(C)$, $|C|\leq f_{\calM,k}(C)$ and $f_{\calM'',k}(C)<|C|$,  we have $|C|=f_{\calM,k}(C)$.
Thus, $C$ is a proper tight set in $G$, which contains $u$.
By Claim~\ref{claim:tight}, this is balanced.
This implies ${\rm sp}_{\calM}(\bR_u)={\rm sp}_{\calM}(\bR_{C})$ for $u\in V(C)$ from the definition of balanced sets,
and $r_{\calM''}(\bR''_{C})=r_{\calM}(\bR_{C})+1$ by $r\notin {\rm sp}_{\cal M}(\bR_u)$.
In total, we obtain $f_{\calM'',k}(C)=f_{\calM,k}(C)$ as  $|\bR''_{C}|-r_{\calM''}(\bR''_{C})=|\bR_{C}|-r_{\calM}(\bR_{C})$.
This however contradicts $f_{\calM'',k}(C)<|C|\leq f_{\calM,k}(C)$, and thus (\ref{eq:claim:G''}) is verified.

Therefore, $(G'',\bR'')$  admits a basic rooted-tree decomposition $\{(T_1,r_1),\dots, (T_t,r_t), (T,r)\}$ by induction.
Define $T_i^*$ by $T_i^*=T_i$ for each $1\leq i\leq t$ with $i\neq j$, and define $T_j^*=T_j\cup T\cup\{uv\}$.
Then, $\{(T_1^*,r_1),\dots, (T_t^*,r_t)\}$ is a basic rooted-tree decomposition since $r$ is parallel to $r_j$ in ${\cal M}''$.

This completes the proof of Theorem~\ref{thm:partition}.
\end{proof}

\section{Algorithms}
\label{sec:alg}
We shall sketch an algorithm for checking the conditions of Theorem~\ref{thm:partition}.
(C1) and (C3) can be obviously checked in polynomial time, provided that the independence oracle of $\calM$ can be implemented in polynomial time.
(C2) can be checked by minimizing the function $f_{\calM,k}':2^{E}\rightarrow \mathbb{Z}$ defined by
\begin{equation*}
f_{\calM,k}'(F)=\begin{cases} 
+\infty & \text{ if } F=\emptyset, \\
f_{\calM,k}(F)-|F|  & \text{ otherwise. }
\end{cases}
\end{equation*}
Namely, (C2) is satisfied if and only if the minimum value of $f_{\calM,k}'$ is non-negative.
Since $f_{{\cal M},k}$ is submodular by Lemma~\ref{lem:sub}, $f_{\calM,k}'$ is an intersecting submodular function.
An intersecting submodular function can be minimized in polynomial time in terms of the size of the ground set and the number of function evaluations (see e.g.,\cite{Schriver,fujishige}).

Here we present an efficient algorithm  via matroid intersection.
The algorithm is based on the idea of Imai~\cite{Imai:1983};
he showed that checking $|F|\leq k|V(F)|-\ell$ for $F\subseteq E$ can be reduced to the problem of computing maximum matchings in auxiliary bipartite graphs.
We extend his technique by reducing to the problem of computing {\em independent matchings}.
For a bipartite graph $H=(V^+,V^-;E)$, suppose there are two matroids ${\cal N}^+=(V^+,{\cal I}^+)$ and ${\cal N}^-=(V^-,{\cal I}^-)$.
A matching $M$ of $H$ is called {\em independent} if $V^+(M)\in {\cal I}^+$ and $V^-(M)\in {\cal I}^-$.
The problem of computing a maximum independent matching is known to be equivalent to the matroid intersection.

Let us briefly sketch a standard algorithm for solving the independent matching problem, following the description given in \cite{murota_book}.
(Although a more efficient algorithm is known~\cite{Cunningham:1986},  the following one is enough for our purpose.)
For an independent matching $M$, consider an auxiliary digraph $\tilde{G}=(\tilde{V},\tilde{A};S^+,S^-)$, so-called the {\em exchangeability graph} with respect to $M$,
consisting of vertex set $\tilde{V}$, edge set $\tilde{A}$, entrance vertex set $S^+$, and exit vertex set $S^-$. 
These are defined as follows:
\begin{align*}
\tilde{V}&=V^+\cup V^-, \qquad \tilde{A}=A^{\circ}\cup M^{\circ}\cup A^+\cup A^-, \\
S^+&=V^+\setminus {\rm sp}_{\cal N^+}(V^+(M)), \qquad S^-=V^-\setminus {\rm sp}_{\cal N^-}(V^-(M)),
\end{align*}
where 
$A^{\circ}$ is a copy of $E$ with direction from $V^+$ to $V^-$,
$M^{\circ}$ is a copy of $M$ with direction from $V^-$ to $V^+$ and
\begin{align*}
A^+&=\{(u,v)\mid u\in V^+(M), v\in {\rm sp}_{\cal N^+}(V^+(M))\setminus V^+(M), V^+(M)-u+v\in {\cal I}^+\} \\
A^-&=\{(v,u)\mid u\in V^-(M), v\in {\rm sp}_{\cal N^-}(V^-(M))\setminus V^-(M), V^-(M)-u+v\in {\cal I}^-\}.
\end{align*}
The algorithm repeatedly constructs the exchangeability graph with respect to the current matching, 
finds an augmenting path (that is, a path from $S^+$ to $S^-$ in the exchangeability graph), and augments through the path.
If no augmenting path exists, then the current matching can be shown to be an optimal solution.
The time to construct the exchangeability graph is $O(|E|+|M||V|Q)$ for each phase and the total computational time becomes $O(r(|E|+r|V|Q))$, where $r$ is the size of a maximum independent matching and $Q$ is 
the time for independence oracle. See \cite{murota_book} for more detail.

With this background, we now show an efficient algorithm for checking (C2).
Let $G=(V,E)$ be a graph with roots $\bR$ and ${\cal M}$ be a matroid on $\bR$ of rank $k$.
We assume throughout the subsequent discussion that $(G,\bR)$ satisfies (C1) and (C3).
We consider an auxiliary graph $G^*=(V,E^*)$, which is obtained by regarding each root $r\in \bR_v$ as a self-loop (i.e., an edge having the same endpoints) attached to $v$.
Let $L$ be the set of these self-loops, and let $E^*=E\cup L$.
Due to the one-to-one correspondence between $\bR$ and $L$, we may think ${\cal M}$ as a matroid on $L$.
For an integer $\ell$ with $0\leq \ell\leq k$, we consider two set functions on $E$ and on $E^*$ defined by
\begin{align*}
f_{k,\ell}(F)&=k|V(F)|-\ell-|\bR_F|+r_{\cal M}(\bR_F) \qquad (F\subseteq E) \\
g_{k,\ell}(F)&=k|V(F)|-\ell+r_{\cal M}(F\cap L) \qquad (F\subseteq E^* ).
\end{align*}
Note $f_{{\cal M},k}=f_{k,k}$.
Also, it is easy to see that $|F|\leq f_{k,\ell}(F)$ holds for any non-empty $F\subseteq E$ if and only if $|F|\leq g_{k,\ell}(F)$ for any non-empty $F\subseteq E^{*}$.
Therefore, in the subsequent discussion, we shall focus on how to check the latter condition.

We first consider the case of $\ell=0$.
We prepare $k$ copies $V^1,\dots, V^k$ of $V$, and a copy $L'$ of $L$.
We define an auxiliary bipartite graph  $H_0=(V_0^+,V_0^-;A_0)$ as follows:
\begin{align*}
V_0^+&=E\cup L, \qquad V_0^-=(\mbox{$\bigcup_{1\leq i\leq k}V^i) \cup L'$}, \\
A_0&=\{(e,e')\mid e'\in L' \text{ is a copy of  } e\in L\} \\
&\ \ \ \ \cup \{(e,v^i)\mid \text{ if $v^i\in V^i$ is a copy of $v\in V$ and $e\in E\cup L$ is incident to $v$ in $G^{*}$}\}. 
\end{align*}
We consider a matroid ${\cal N}^-$ on $V_0^-$, which is the direct sum of ${\cal M}$ on $L'$ and the free matroid on $\bigcup_i V^i$.
We also consider the free matroid ${\cal N}^+$  on $V_0^+$. 
The following claim is immediate from  Rado's theorem.
\begin{lemma}
$|F|\leq g_{k,0}(F)$ holds for any $F\subseteq E^{*}$ if and only if 
$H_0$ has an independent matching covering $V_0^+$.
\end{lemma}
\begin{proof}
Rado's theorem (see e.g.,\cite{murota_book}) implies that 
the size of a maximum independent matching is equal to 
\begin{equation}
\label{eq:rado}
\min\{r_{{\cal N}^-}(\Gamma(F))+|V_0^+\setminus F| \mid F\subseteq V_0^+ \},   
\end{equation}
where $\Gamma(F)$ denotes the set of neighbors of $F$ in $H_0$.
Notice $r_{\cal N^-}(\Gamma(F))+|V_0^+\setminus F|=k|V(F)|+r_{\cal M}(L'\cap F)+|(E\cup L)\setminus F|$ for any $F\subseteq V_0^+=E\cup L$.
Therefore, $|F|\leq g_{k,0}(F)$ holds for any $F\subseteq E$ if and only if the size of a maximum independent matching is equal to $|E\cup L|$.  
\end{proof}

Let us analyze the time complexity.
Let $Q$ be the time of independence oracle of ${\cal M}$.
The size of $H_0$ is $O(k^2|V|)$ by (C3). 
Notice also that, since ${\cal N}^+$ is free and ${\cal N}^-$ is the direct sum of ${\cal M}$ on $L'$ and the free matroid on $\bigcup_i V^i$,
the exchangeability graph satisfies $A^+=\emptyset$ and $A^-\subseteq L\times L$  for any independence matching $M$.
Therefore we can construct the exchangeability graph $\tilde{G}$ and find a path from $S^+$ to $S^-$ in $O(k^2|V|+|L|^2Q)$ time. 
The total time complexity thus becomes $O(k|V|(k^2|V|+|L|^2Q))$ since we have $O(|E\cup L|)=O(k|V|)$ iterations. 


Checking $|F|\leq g_{k,\ell}(F)$ for general $\ell$ can be performed by extending the idea above.
Take an edge $e\in E$, and prepare $\ell$ copies $e_1,\dots, e_{\ell}$ of $e$.
We consider an auxiliary bipartite graph $H_e=(V_e^+,V_e^-;A_e)$ defined by
\begin{align*}
V_e^+&=V_0^+\cup\{e_1,\dots, e_{\ell} \}, \qquad V_e^-=V_0^-, \\
A_e&=A_0\cup \{(e_j,v^i)\mid \text{ if $v^i\in V^i$ is a copy of $v\in V$ and  $e$ is incident to $v$ in $G^{*}$} \}.
\end{align*}
Then, the exactly same argument can be applied to show the following:
\begin{lemma}
Suppose $|F|\leq g_{k,0}(F)$ for any $F\subseteq E^{*}$.
Then, $|F|\leq g_{k,\ell}(F)$ holds for any $F\subseteq E^{*}$ with $e\in F$ if and only if 
$H_e$ has an independent matching covering $V_e^+$.
\end{lemma}
Thus, if we check the size of a maximum independent matching in $H_e$ for every $e\in E$,
we can decide whether $|F|\leq g_{k,\ell}(F)$ for any non-empty $F\subseteq E^{*}$.
(Note that, if $|F|>g_{k,\ell}(F)$ for some $F\subseteq E^*$, then $F\cap E\neq \emptyset$ by (C1).)
Since a maximum independent matching of $H_e$ can be computed from that of $H_0$ by $\ell$ augmentations,
the additional time we need is  $O(\ell (k^2|V|+|L|^2Q))$.
Since we need to check it for every $e\in E$, the total computational time amounts to $O(\ell k|V|(k^2|V|+|L|^2Q))$. 
Consequently, we obtain the following.
\begin{theorem}
\label{thm:alg}
Let $G=(V,E)$ be a graph with a multiset $\bR$ of vertices and ${\cal M}$ be a matroid on $\bR$ of rank $k$.
Then, one can check whether $(G,\bR)$ satisfies the conditions of Theorem~\ref{thm:partition} in $O(k^4|V^2|+k^2|V||\bR|^2Q)$ time,
where $Q$ is the time of independent oracle of ${\cal M}$.
\end{theorem}

\noindent 
{\bf  Remark.} 
A {\em Dulmage-Mendelsohn-type decomposition} is also known for the independent matching problem.  
This decomposition for $H_e$ gives all information on tight sets (defined in the previous section) containing $e$, 
and we can efficiently find, say, a maximal tight set containing $e$ using the exchangeability graph with respect to a maximum matching 
(see~\cite[Chapter 2]{murota_book} for more detail).
Similarly, if $G$ does not satisfy the counting condition, the decomposition of $H_e$ shows maximal violating sets containing $e$.
Note that our proof of Theorem~\ref{thm:partition} is constructive, provided that we can detect a violating set if $G$ violates the counting condition.
We can thus explicity find a basic decomposition in polynomial time. 

%
%

\section{Dual Form of Theorem~\ref{thm:partition}}
\label{sec:dual}
In this section, we present a dual form of Theorem~\ref{thm:partition} which generalizes Tutte-Nash-Williams tree-packing theorem.
Extending the notion of rooted-trees, a pair $(C,r)$ of $C\subseteq E$ and $r\in V$ is called a {\em rooted-component} 
if either (i) $C=\emptyset$ or (ii) $C$ is connected and $r\in V(C)$.
Let $V(C,r)=V(C)\cup \{r\}$.
Also, for $X\subseteq V$, let $\bR_X$ be the multiset $\{r_i\in \bR\mid r_i\in X\}$.
\begin{theorem}
\label{thm:dual}
Let $G=(V,E)$ be a graph, $\bR=\{r_1,\dots, r_t\}$ be a multiset of vertices, $\calM$ be a matroid on $\bR$ of rank $k$ and the rank function $r_{\calM}:2^{\bR}\rightarrow \mathbb{Z}$.
Then, $(G,\bR)$ can be decomposed into rooted-components $(C_1,r_1),\dots, (C_t,r_t)$ such that the multiset $\{r_i\in\bR\mid v\in V(C_i,r_i)\}$ is a spanning set of $\calM$ for every $v\in V$ 
if and only if
\begin{equation}
\label{eq:dual}
|\delta_{G}({\cal P})|\geq k|{\cal P}|-\sum_{X\in {\cal P}}r_{\calM}(\bR_X)
\end{equation}
for every partition ${\cal P}$ of $V$ into non-empty subsets.
\end{theorem}

Theorem~\ref{thm:dual} follows from a standard argument based on an explicit formula of the rank function of ${\cal N}(f_{\calM,k})$ given in Theorem~\ref{thm:rank} below.
The following lemma indicates a reason why the rank function can be described in such a simple way.
\begin{lemma}
\label{lem:disjoint}
Let $G=(V,E)$ be a graph $\bR$ be a multiset of vertices, and ${\cal M}$ be a matroid of rank $k$ and the rank function $r_{\calM}$.
Suppose (C1) is satisfied.
Then, for any $F_1,F_2\subseteq E$ with $V(F_1)\cap V(F_2)\neq \emptyset$, 
$f_{\calM,k}(F_1)+f_{\calM,k}(F_2)\geq f_{\calM,k}(F_1\cup F_2)$ holds.
\end{lemma}
\begin{proof}
From $V(F_1\cup F_2)=V(F_1)\cup V(F_2)$, 
$\bR_{F_1\cup F_2}=\bR_{F_1}\cup \bR_{F_2}$, and the submodularity of $r_{\calM}$, it easily follows that
$f_{\calM,k}(F_1)+f_{\calM,k}(F_2)-f_{\calM,k}(F_1\cup F_2) 
\geq k|V(F_1)\cap V(F_2)|-k-|\bR_{F_1}\cap \bR_{F_2}|+r_{\calM}(\bR_{F_1})+r_{\calM}(\bR_{F_2})-r_{\calM}(\bR_{F_1}\cup \bR_{F_2})
\geq k|V(F_1)\cap V(F_2)|-k-|\bR_{F_1}\cap \bR_{F_2}|+r_{\calM}(\bR_{F_1}\cap \bR_{F_2})$.
Let $s=r_{\calM}(\bR_{F_1}\cap \bR_{F_2})$.
By (C1), we have $|\bR_v|\leq s$ for every $v\in V(F_1)\cap V(F_2)$.
We thus have $k|V(F_1)\cap V(F_2)|-k-|\bR_{F_1}\cap \bR_{F_2}|+r_{\calM}(\bR_{F_1}\cap \bR_{F_2})=\sum_{v\in V(F_1)\cap V(F_2)} (k-|\bR_v|)-k+s\geq 0$,
where the last inequality follows from  $|V(F_1)\cap V(F_2)|\geq 1$ and $|\bR_v|\leq s$.
\end{proof}

For $F\subseteq E$ and a partition ${\cal P}$ of $V$, $\delta_F({\cal P})$ denotes the subset of $F$ connecting two distinct components of ${\cal P}$.
\begin{theorem}
\label{thm:rank}
Let $(G=(V,E),\bR)$ be a graph with roots, and ${\cal M}$ be a matroid on $\bR$ of rank $k$ and the rank function $r_{\calM}$.
Suppose (C1) is satisfied.
Then, the rank of $F\subseteq E$ in ${\cal N}(f_{\calM,k})$ is equal to
\begin{equation}
\label{eq:rank}
\min \{|\delta_F({\cal P})|+k(|V|-|{\cal P}|)-|\bR|+\sum_{X\in {\cal P}} r_{\calM}(\bR_X) \},
\end{equation}
where the minimum is taken over all partitions ${\cal P}$ of $V$ into non-empty subsets.
\end{theorem}
\begin{proof}
By Proposition~\ref{prop:rank}, the rank of $F$ in ${\cal N}(f_{\calM,k})$ is equal to
\begin{equation}
\label{eq:form1}
\min\{|F_0|+\sum_{1\leq i\leq m}f_{\calM,k}(F_i)\},
\end{equation}
where the minimum is taken over all partitions $\{F_0,F_1,\dots, F_m\}$ of $F$ such that $F_i\neq \emptyset$ for each $i=1,\dots,m$ (and $F_0$ may be empty).
Let $\{F_0,F_1,\dots, F_m\}$ be a minimizer of (\ref{eq:form1}) such that $m$ is smallest among all minimizers.
Let $X_i=V(F_i)$ for $i=1,\dots, m$ and let $X_v=\{v\}$ for every $v\in V\setminus V(F)$.
We set ${\cal P}_1=\{X_i\mid 1\leq i\leq m\}$, ${\cal P}_2=\{X_v\mid v\in V\setminus V(F)\}$, and ${\cal P}={\cal P}_1\cup {\cal P}_2$.
By the minimality of $m$ and Lemma~\ref{lem:disjoint}, ${\cal P}_1$ is a partition of $V(F)$ and hence ${\cal P}$ is a partition of $V$.
Also, if $uv\in F_0$  satisfies $u\in V(F_j)$ and $v\in V(F_j)$ for some $j$ with $1\leq j\leq m$, then we have
$|F_0|+\sum_{i}f_{\calM,k}(F_i)>|F_0-uv|+f_{\calM,k}(F_j+uv)+\sum_{i\neq j}f_{\calM,k}(F_i)$ by $f_{\calM,k}(F_j)=f_{\calM,k}(F_j+uv)$,
which contradicts that  $\{F_0,F_1,\dots, F_m\}$ is a minimizer of (\ref{eq:form1}),
Thus for each $uv\in F_0$ there is no $i$ such that $V(F_i)$ contains both $u$ and $v$, implying $F_0=\delta_F({\cal P})$.
Also, since each component of ${\cal P}_2$ consists of a single vertex of $V\setminus V(F)$, we clearly have
\begin{align*}
&k(|V\setminus V(F)|-|{\cal P}_2|)=0 \\
&\sum_{X\in {\cal P}_2}(|\bR_X|-r_{\calM}(\bR_X))=\sum_{v\in V\setminus V(F)}(|\bR_v|-r_{\calM}(\bR_v))=0.
\end{align*}
In total, the rank of $F$ is equal to
\begin{align*}
|F_0|+\sum_{1\leq i\leq m}f_{\calM,k}(F_i)
&=|\delta_F({\cal P})|+\sum_{X\in {\cal P}_1} (k|X|-k-|\bR_X|+r_{\cal M}(\bR_X)) \\
&=|\delta_F({\cal P})|+k(|V(F)|-|{\cal P}_1|)-\sum_{X\in {\cal P}_1} (|\bR_X|-r_{\cal M}(\bR_X)) \\
&=|\delta_F({\cal P})|+k(|V|-|{\cal P}|)-\sum_{X\in {\cal P}} (|\bR_X|-r_{\cal M}(\bR_X)) \\
&=|\delta_F({\cal P})|+k(|V|-|{\cal P}|)-|\bR|+\sum_{X\in {\cal P}} r_{\calM}(\bR_X),
\end{align*}
and hence (\ref{eq:rank}) is at most the rank of $F$.

To see the converse direction, consider a partition ${\cal P}=\{X_1,\dots, X_s\}$ of $V$ into non-empty subsets.
Let $F_0=\delta_F({\cal P})$ and $F_i=\{uv\in F\mid  u\in X_i, v\in X_i\}$ for $1\leq i\leq s$.
Then, $\{F_0,F_1,\dots, F_s\}$ is a partition of $F$.
Note that, for any $X_i\in {\cal P}$, we have $k(|X_i|-1)-|\bR_{X_i}|+r_{\cal M}(\bR_{X_i})\geq 0$.
Thus, we have
\begin{align*}
&|\delta_F({\cal P})|+k(|V|-|{\cal P}|)-|\bR|+\sum_{X\in {\cal P}} r_{\calM}(\bR_X) \\
&=|\delta_F({\cal P})|+\sum_{X_i\in {\cal P}} [k(|X_i|-1)-|\bR_{X_i}|+r_{\cal M}(\bR_{X_i})] \\
&\geq |\delta_F({\cal P})|+\sum_{i:F_i\neq \emptyset} [k(|X_i|-1)-|\bR_{X_i}|+r_{\cal M}(\bR_{X_i})] \\
&=|F_0|+\sum_{i:F_i\neq \emptyset} f_{\calM,k}(F_i)
\end{align*}
and hence (\ref{eq:rank}) is no less than the rank of $F$.
\end{proof}

\begin{proof}[Proof of Theorem~\ref{thm:dual}]
(Necessity:)
Suppose $(G,\bR)$ admits a decomposition $(C_1, r_1),\dots, (C_t,r_t)$ such that $\{r_i\in \bR\mid v\in V(C_i,r_i)\}$ is a spanning set of ${\cal M}$ for each $v\in V$.
Since $C_i$ is connected with $r_i\in V(C_i)$, we can assign an orientation of each edge of $C_i$ such that each vertex of $V(C_i)\setminus \{r_i\}$ has in-degree at least one.
Suppose we orient all of $C_i$ in such a way.
Observe then that, for every $X\in {\cal P}$,
the sum of $r_{\cal M}(\bR_X)$ and the number of arcs entering to $X$ must be at least $k$ because the multiset $\{r_i\in \bR\mid v\in V(C_i,r_i)\}$ has rank $k$ for every $v\in X$.
This however implies $|\delta_G({\cal P})|+\sum_{X\in {\cal P}}r_{\cal M}(\bR_X)\geq k|{\cal P}|$ and ${\cal P}$ satisfies (\ref{eq:dual}). 

(Sufficiency:)
Suppose (\ref{eq:dual}) is satisfied.
Take a maximal multi-subset $\bR'$ of $\bR$ that satisfies (C1).
Without loss of generality, we denote $\bR'=\{r_1,\dots, r_{t'}\}$.
From the maximality of $\bR'$, $r_{\calM}(\bR'_X)=r_{\calM}(\bR_X)$ holds for any $X\subseteq V$.
Using (\ref{eq:dual}) we obtain
\begin{align*}
&|\delta_G({\cal P})|+k(|V|-|{\cal P}|)-|\bR'|+\sum_{X\in {\cal P}} r_{\calM}(\bR'_X) \\
&=|\delta_G({\cal P})|+k(|V|-|{\cal P}|)-|\bR'|+\sum_{X\in {\cal P}} r_{\calM}(\bR_X) \geq k|V|-|\bR'|
\end{align*}
for every partition ${\cal P}$ of $V$.
Theorem~\ref{thm:rank} thus implies that the rank of ${\cal N}(f_{\calM|\bR',k})$ is equal to $k|V|-|\bR'|$,
and hence, taking a base $E'$ of ${\cal N}(f_{\calM|\bR',k})$, we obtain a subgraph $(G'=(V,E'),\bR')$ of $(G,\bR')$ that satisfies (C1)(C2)(C3) with respect to $\calM|\bR'$.
By Theorem~\ref{thm:partition}, $(G',\bR')$ admits a basic rooted-tree decomposition $(T_1,r_1), \dots, (T_{t'},r_{t'})$.
For each $uv\in E\setminus E'$ there clearly exists at least one rooted-tree $(T_i,r_i)$ with $u\in V(T_i,r_i)$ or $v\in V(T_i,r_i)$.
We add $uv$ to (arbitrary one of) such a $(T_i,r_i)$.
We then obtain  a desired rooted-component decomposition.
\end{proof}

\section{Body-bar Frameworks with Boundaries}
\label{sec:body_bar}
We now move to applications of Theorem~\ref{thm:partition} to rigidity theory.
This section concerns with body-bar frameworks, which are structures consisting of rigid bodies articulated by bars as shown in Figure~\ref{fig:body_bar}.
In particular we propose extensions of Tay's combinatorial characterization for infinitesimal rigidity of body-bar frameworks to those with bar-boundary and pin-boundary, 
where some of bodies are linked to the external fixed environment by bars or pins as shown in Figure~\ref{fig:body_bar}(b)(c).

We begin with introducing necessary terminology from geometry, and then we review Tay's combinatorial characterization in Subsection~\ref{subsec:body_bar}.
In Subsection~\ref{subsec:body_bar_bar} we shall discuss body-bar frameworks with bar-boundary and present an extension of Tay's result (Theorem~\ref{thm:body_bar}) based on basic rooted-tree decompositions.
In Subsection~\ref{subsec:body_bar_pin} we present an extension of Tay's result to pinned body-bar frameworks by reducing them to the bar-boundary case.

\subsection{Grassmannian}
\label{subsec:grassmannian}
Throughout the subsequent discussion we use following notation.
The homogenous coordinate of a point in the real projective space $\mathbb{P}^d$ is written by $[p]$, that is, the ratio of the coordinates of $p\in \mathbb{R}^{d+1}\setminus\{0\}$.
Conversely for $p\in \mathbb{R}^d$ the corresponding  homogenous coordinate is denoted by $[p,1]$ by a canonical embedding of $\mathbb{R}^d$ to $\mathbb{P}^d$.
Also we simply denote $D={d+1\choose 2}$.

Recall that the exterior product $\bigwedge^k \mathbb{R}^{d+1}$ of degree $k$ is a ${d+1 \choose k}$-dimensional vector space.
In particular, we may identify $\bigwedge^2 \mathbb{R}^{d+1}$ with $\mathbb{R}^D$.
The standard Euclidean inner product will be used throughout the paper.
Also, for $a\in \bigwedge^k \mathbb{R}^{d+1}$ and $b\in \bigwedge^{d+1-k} \mathbb{R}^{d+1}$, 
the inner product $\langle a, \ast b\rangle$ of $a$ and the Hodge dual  $\ast b$ of $b$ is simply denoted by $\langle a, b\rangle$ 
(see e.g.,~\cite{crapo1991} where $\ast$ is called the Hodge star complement).

The collection of $k$-dimensional subspaces in $\mathbb{R}^{d+1}$ is called the {\em Grassmannian}, denoted $Gr(k,\mathbb{R}^{d+1})$.
The {\em Pl{\"u}cker embedding} $p^*:Gr(k,\mathbb{R}^{d+1})\rightarrow \mathbb{P}(\bigwedge^k \mathbb{R}^{d+1})$ is a bijection
between $k$-dimensional vector spaces $X\in Gr(k,\mathbb{R}^{d+1})$ and 
projective equivalence classes $[v_1\wedge \dots \wedge v_k]\in \mathbb{P}(\bigwedge^k \mathbb{R}^{d+1})$ of decomposable elements,
where $\{v_1,\dots, v_k\}$ is a basis of $X$.
In the subsequent discussions, we shall identify $Gr(k,\mathbb{R}^{d+1})$ with its image of the Pl{\"u}cker embedding,
and regard $Gr(k,\mathbb{R}^{d+1})$ as a subset of $\mathbb{P}(\bigwedge^k \mathbb{R}^{d+1})$.
Thus  a $k$-dimensional linear subspace $X\in Gr(k,\mathbb{R}^{d+1})$ of $\mathbb{R}^{d+1}$ is sometimes 
referred to as a point in $\mathbb{P}(\bigwedge^k \mathbb{R}^{d+1})$ if it is clear from the context.
Also note $\mathbb{P}^d=Gr(1,\mathbb{R}^{d+1})$.

It is well-known that each point of $Gr(k,\mathbb{R}^{d+1})$ can be coordinatized by the so-called {\em Pl{\"u}cker coordinate} once we fix  a basis of $\mathbb{R}^{d+1}$. 
We shall use the standard basis ${\bm e}_1,\dots, {\bm e}_{d+1}$ of $\mathbb{R}^{d+1}$.
If a basis $\{v_1,\dots, v_k\}$ of $X\in Gr(k,\mathbb{R}^{d+1})$ is represented by $v_i=\sum_{j=1}^{d+1}p_{ij}{\bm e}_j$ 
with the $k\times (d+1)$-matrix $P=[p_{ij}]$,
then we have 
\[
v_1\wedge \dots \wedge v_k=\mbox{$\sum_{i_1<\cdots<i_k}$} \det P_{i_1,\dots, i_k} {\bm e}_{i_1}\wedge \dots \wedge {\bm e}_{i_k}, 
\]
where $P_{i_1,\dots, i_k}$ is the $k\times k$-submatrix of $P$ consisting of $i_j$-th columns.
The ratio of $\det P_{i_1,\dots, i_k}$ for $1\leq i_1< \dots < i_k\leq d+1$ is called {\em the Pl{\"u}cker coordinate} of $X$. 


\subsection{Body-bar Frameworks}
\label{subsec:body_bar}
In the context of infinitesimal rigidity, a $d$-dimensional {\em body-bar framework} is customarily denoted by a pair $(G,{\bm b})$, where
\begin{itemize}
\item $G=(V,E)$ is a graph;
\item ${\bm b}$ is a bar-configuration, that is, a mapping,
\begin{equation}
\label{eq:bar_configuration}
\begin{split}
{\bm b}:E &\rightarrow Gr(2,\mathbb{R}^{d+1}) \\
e &\mapsto [{\bm b}_e].
\end{split}
\end{equation}
\end{itemize}
Namely, each vertex corresponds to a body, and each edge corresponds to a bar connecting two bodies.
Note that for analyzing infinitesimal rigidity, we only need to know the direction of each bar, which is  specified by ${\bm b}$, (see Appendix~\ref{app:bar_description}).

An {\em infinitesimal motion} of $(G,{\bm b})$ is a mapping ${\bm m}:V\rightarrow \bigwedge^{d-1} \mathbb{R}^{d+1}$ satisfying the first-order length constraint by bars:
\begin{align}
\langle {\bm m}(u)-{\bm m}(v), {\bm b}_e\rangle &=0 && \text{ for } e=uv \in E. \label{eq:body_bar_bar_const}
\end{align}
A detailed geometric meaning of (\ref{eq:body_bar_bar_const}) is explained in Appendix~\ref{app:bar_description}.
(Detailed description can be also found in e.g.,~\cite{crapo:1982,white:whiteley:87}.) 
Since $\bigwedge^{d-1}\mathbb{R}^{d+1}$ is a $D$-dimensional real vector space,
the motion space is a linear subspace of $\mathbb{R}^{D|V|}$.
An infinitesimal motion ${\bm m}$ is called {\em trivial} if ${\bm m}(v)={\bm m}(u)$ for every $u,v\in V$.
$(G,{\bm b})$ is {\em infinitesimally rigid} if every possible motion is trivial.
$(G,{\bm b})$ is called {\em minimally infinitesimally rigid} if removing any bar results in a framework that is not infinitesimally rigid.

Tay~\cite{tay:84} proved that, for almost all bar-configurations ${\bm b}$, $(G,{\bm b})$ is minimally infinitesimally rigid  if and only if 
 $|E|=D|V|-D$ and $|F|\leq D|V(F)|-D$ for any non-empty $F\subseteq E$
or equivalently,  $G$ contains $D$ edge-disjoint spanning trees by Nash-Williams' theorem.
In the next paragraph, we shall provide an extension of this result based on rooted-tree decompositions.

\subsection{Body-bar Frameworks with Bar-boundary}
\label{subsec:body_bar_bar}
A $d$-dimensional {\em body-bar framework with bar-boundary}
is defined as a tuple $(G,\bR; {\bm b},{\bm b}^{\circ})$, where
\begin{itemize}
\item $G=(V,E)$ is a graph and $\bR$ is a multiset of vertices;
\item ${\bm b}$ is a bar-configuration given in (\ref{eq:bar_configuration});
\item ${\bm b}^{\circ}$ is a configuration of bar-boundary, that is, a mapping,
\begin{align*}
{\bm b}^{\circ}:\bR &\rightarrow Gr(2,\mathbb{R}^{d+1}) \\
r &\mapsto [{\bm b}^{\circ}_r]. 
\end{align*}
\end{itemize}
Namely, along with a conventional body-bar framework $(G,{\bm b})$, we introduce abstract signs $\bR$ of bar-boundary and its realization ${\bm b}^{\circ}$ 
in such a way that the body corresponding to a vertex $v\in V$ is linked to the fixed external environment by a bar ${\bm b}^{\circ}(r)$, for each $r\in \bR_v$.

An {\em infinitesimal motion} of $(G,\bR;{\bm b},{\bm b}^{\circ})$ is a mapping ${\bm m}:V\rightarrow \bigwedge^{d-1} \mathbb{R}^{d+1}$ satisfying not only bar-constraints (\ref{eq:body_bar_bar_const}) but also 
boundary-constraints:
\begin{equation}
\langle {\bm m}(v), {\bm b}^{\circ}_r\rangle =0 \qquad \text{ for } r\in \bR_v \text{ with } v\in V. \label{eq:body_bar_external}
\end{equation}
This condition can be obtained by setting ${\bm m}(u)=0$ in (\ref{eq:body_bar_bar_const}).
The  set of possible infinitesimal motions forms a linear subspace of $\mathbb{R}^{D|V|}$, and
$(G,\bR;{\bm b},{\bm b}^{\circ})$ is said to be {\em infinitesimally rigid} if there is no nonzero motion.
So in this case we do not allow even trivial motions.
$(G,\bR;{\bm b},{\bm b}^{\circ})$ is {\em minimally infinitesimally rigid} if removing any bar (including boundary-bar) results in a flexible framework.

Theorem~\ref{thm:body_bar_bar} below shows a  combinatorial characterization of body-bar frameworks with ``non-generic'' bar-boundary.
The proof is based on Theorem~\ref{thm:partition}, and the proof idea is from Whiteley~\cite{whiteley:88}.
\begin{theorem}
\label{thm:body_bar_bar}
Let $G=(V,E)$ be a graph, $\bR$ be a multiset of vertices,  and  ${\bm b}^{\circ}:\bR\rightarrow Gr(2,\mathbb{R}^{d+1})$.
Then, there exists a bar-configuration ${\bm b}:E\rightarrow Gr(2,\mathbb{R}^{d+1})$ such that
the body-bar framework $(G,\bR; {\bm b},{\bm b}^{\circ})$ is minimally infinitesimally rigid if and only if
\begin{itemize} 
\item $\{{\bm b}^{\circ}_r\mid r\in \bR_v\}$ is linearly independent for each $v\in V$,
\item $|F|+|\bR_F|\leq D|V(F)|-D
+{\rm dim }(\{{\bm b}^{\circ}_r\mid r\in \bR_F\})$ for any non-empty $F\subseteq E$.
\item $|E|+|\bR|=D|V|$.
\end{itemize}
\end{theorem}
\begin{proof}
(``If''-part:)
Suppose $G$ satisfies the above counting conditions.
Let ${\cal M}$ be a linear matroid on $R$ represented by vectors ${\bm b}^{\circ}_r \ (r\in \bR)$, and $k$ denotes the rank of ${\cal M}$.
Let us first check $k=D$. Since $\dim (\bigwedge^2 \mathbb{R}^{d+1})=D$, we have $k\leq D$. 
On the other hand, from the counting condition, we have $D|V|=|E|+|\bR|\leq D|V|-D+k$, implying $k=D$.  

Thus, by Theorem~\ref{thm:partition}, 
$E$ admits a basic rooted-tree decomposition $\{(T_1,r_1), \dots, (T_t,r_t)\}$ with respect to ${\cal M}$.
We define a bar-configuration ${\bm b}$ on $E$ by
\[
 {\bm b}(e)={\bm b}^{\circ}(r_i) \qquad \text{for } e\in T_i.
\]
Let us check that $(G,\bR;{\bm b},{\bm b}^{\circ})$ is indeed infinitesimally rigid.

Let ${\bm m}$ be an arbitrary infinitesimal motion, and let us show ${\bm m}(v)=0$ for each  $v\in V$.
There are exactly $D$ rooted-trees that span $v$,  and without loss of generality we denote them by $(T_{1},r_{1}),\dots, (T_{D},r_{D})$.
For each $1\leq i\leq D$, $T_{i}$ contains a unique path $r_{i}=v_1,v_2,\dots, v_s=v$
from $r_{i}$ to $v$.
Hence, by (\ref{eq:body_bar_bar_const}) and (\ref{eq:body_bar_external}),
we have $\langle {\bm m}(v_j)-{\bm m}(v_{j+1}), {\bm b}^{\circ}_{r_{i}}\rangle =0$ for $j=1,\dots, s$ and ${\langle \bm m}(v_1), {\bm b}^{\circ}_{r_{i}}\rangle =0$.
Summing up these equations, we obtain  $\langle {\bm m}(v_s), {\bm b}^{\circ}_{r_i}\rangle =\langle {\bm m}(v), {\bm b}^{\circ}_{r_{i}}\rangle =0$ for each $1\leq i\leq D$.
Since the decomposition is basic with respect to ${\cal M}$, $\{{\bm b}^{\circ}_{r_{i}}\mid 1\leq i\leq D\}$ is linearly independent.
We thus obtain  ${\bm m}(v)=0$ for every $v\in V$.
In other words $(G,R,{\bm b},{\bm b}^{\circ})$ is infinitesimally rigid.
The minimality is straightforward because the space of ${\bm m}$ is $D|V|$-dimensional while there are only $D|V|$ linear equations (\ref{eq:body_bar_bar_const}) (\ref{eq:body_bar_external}) by the third condition.

\medskip

(``Only-if''-part:)
Because of the minimality, the first condition is clearly necessary.
Since the space of a mapping $V\rightarrow \bigwedge^2\mathbb{R}^{d+1}$ is $D|V|$-dimensional, the third condition is also necessary for the minimal rigidity.
To see the second condition, consider the sub-framework induced by $F\subseteq E$, that is, a realization of a graph $((V,F),\bR_F)$ with roots.
Then, clearly, this sub-framework has $D|V\setminus V(F)|$ independent motions 
since each body associated with $v\in V\setminus V(F)$ has no connection to the other bodies in this sub-framework.
We also have at least $D-\dim (\{{\bm b}^{\circ}(r)\mid r\in \bR_F\})$ independent motions  since the component consisting of the bodies of $V(F)$ has independent motions in the orthogonal complement of 
$\{{\bm b}^{\circ}_r\mid r\in \bR_F\}$.
Thus the number of independent linear equations in the sub-framework is upper bounded by $D|V|-D|V\setminus V(F)|-(D-\dim (\{ {\bm b}^{\circ}(r)\mid r\in \bR_F \}))=D|V(F)|-D+\dim (\{ {\bm b}^{\circ}(r)\mid r\in \bR_F \})$,
and the second condition is necessary for minimality.
\end{proof}

Combining Theorem~\ref{thm:dual} and Theorem~\ref{thm:body_bar_bar}, we immediately obtain the dual form.
\begin{corollary}
\label{cor:body_bar_dual}
Let $G$ be a graph, $\bR$ be a multiset of vertices, and ${\bm b}^{\circ}:\bR\rightarrow Gr(2,\mathbb{R}^{d+1})$.
Then, there exists a bar-configuration ${\bm b}:\bR\rightarrow Gr(2,\mathbb{R}^{d+1})$ such that 
the body-bar framework $(G,\bR;{\bm b},{\bm b}^{\circ})$ is infinitesimally rigid if and only if 
\begin{equation}
|\delta_G({\cal P})|\geq D|{\cal P}|-\sum_{X\in {\cal P}} \dim(\{{\bm b}^{\circ}_r\mid r\in \bR_X\})
\end{equation}
for every  partition ${\cal P}$ of $V$.
\end{corollary}

\noindent
{\bf Remark.}
The set of bar-configurations ${\bm b}$ for which the dimension of motion space is minimized forms a dense subset of the set of all possible bar-configurations ${\bm b}$, see, e.g., \cite{whiteley:88,tanigawa2010}. 
Theorem~\ref{thm:body_bar_bar} (resp.~Corollary~\ref{cor:body_bar_dual}) hence implies
the necessary and sufficient condition for the infinitesimal rigidity of $(G,\bR;{\bm b},{\bm b}^{\circ})$ 
for almost all bar-configurations ${\bm b}$.

\subsection{Pinned Body-bar Frameworks}
\label{subsec:body_bar_pin}
A $d$-dimensional {\em pinned body-bar framework} is defined as $(G,\bR; {\bm b},{\bm p}^{\circ})$, where
\begin{itemize}
\item $G=(V,E)$ is a graph and $\bR$ is a multiset of vertices;
\item ${\bm b}$ is a bar-configuration given in (\ref{eq:bar_configuration});
\item ${\bm p}^{\circ}:\bR\rightarrow \mathbb{R}^{d}$ is a configuration of pin-boundary.
\end{itemize}
Namely, $\bR$ denotes abstract signs of pinning and their positions are specified by ${\bm p}^{\circ}$ 
in such a way that the body corresponding to a vertex $v\in V$ is pinned at $\bp^{\circ}(r)$ for each $r\in \bR_v$.
(Note that each body may be pinned at more than one point.)



An {\em infinitesimal motion} of $(G,\bR;{\bm b},{\bm p}^{\circ})$ is a mapping ${\bm m}:V\rightarrow \bigwedge^{d-1}\mathbb{R}^{d+1}$ satisfying bar-constraints (\ref{eq:body_bar_bar_const})
and pin-boundary constraints, which can be written by, for each $v\in V$  and $r\in \bR_v$,
\begin{equation*}
\langle {\bm m}(v), (\bp^{\circ}(r), 1) \wedge (q,1)\rangle=0 \qquad \text{ for any } q\in \mathbb{R}^d
\end{equation*}
where (the ratio of) $(\bp^{\circ}(r), 1) \wedge (q,1)$ corresponds to the Pl{\"u}cker coordinate of the line passing through $\bp^{\circ}(r)$ and $q$.
This definition is justified by observing that pinning a body at a point $p$ is equivalent to linking $p$ by bars with the fixed external environment.
Thus, pinned body-bar frameworks can be considered as a special case of body-bar frameworks with bar-boundary.
As before, $(G,\bR;{\bm b},{\bm p}^{\circ})$ is said to be {\em infinitesimally rigid} if there is no nonzero motion.
It is now straightforward to derive the following combinatorial characterization due to Corollary~\ref{cor:body_bar_dual}. 
\begin{theorem}
\label{thm:body_bar}
Let $(G,\bR)$ be a graph with roots, and let ${\bm p}^{\circ}:\bR\rightarrow \mathbb{R}^d$.
Then, there exists ${\bm b}:E\rightarrow Gr(2,\mathbb{R}^{d+1})$ such that
$(G,\bR; {\bm b},{\bm p}^{\circ})$ is infinitesimally rigid if and only if
\begin{equation}
|\delta_G({\cal P})|\geq D|{\cal P}|-\sum_{X\in {\cal P}, \bR_X\neq\emptyset} \sum_{i=1}^{d_X+1} (d-i+1)
\end{equation}
for every  partition ${\cal P}$ of the vertex set $V$, 
where $d_X$ denotes the dimension of the affine span of ${\bm p}^{\circ}(\bR_X)$.
%
\end{theorem}
\begin{proof}
Recall that, 
for any set $P$ of points, the dimension of the linear span of 
$\bigcup_{p\in P} \{(p,1)\wedge (q,1) \mid  q\in \mathbb{R}^d\}$ is equal to $\sum_{i=1}^{d_P+1} (d-i+1)$,
where $d_P$ denotes the dimension of the affine span of $P$.
Since pinning a body at a point $p\in \mathbb{R}^d$ is equivalent to adding bar-constraints between  $p$ and the external environment, the statement directly follows from Corollary~\ref{cor:body_bar_dual}.
\end{proof}

\section{Bar-joint Frameworks with Boundary}
\label{sec:bar_joint}
We now proceed to the rigidity of 2-dimensional bar-joint frameworks.
As in the previous section we first review frameworks without boundary, and then move to models with boundary.
 
\subsection{2-dimensional Bar-joint Frameworks}
\label{subsec:bar_joint}
For a graph $G=(V,E)$,  an injective mapping $\bp:V\rightarrow \mathbb{R}^2$ is called a {\em joint-configuration}.
A {\em 2-dimensional bar-joint framework} is defined as a pair $(G,{\bm p})$ of a graph $G=(V,E)$ and a joint-configuration $\bp$.
An infinitesimal motion of the framework is customarily defined by a mapping $\bmm:V\rightarrow \mathbb{R}^2$ such that 
\begin{equation}
\label{eq:bar_joint_motion1}
\langle \bp(u)-\bp(v), \bmm(u)-\bmm(v)\rangle =0  \qquad \text{for } uv\in E.
\end{equation}

Note that an infinitesimal isometry of $\mathbb{R}^2$ induces a nonzero motion of $(G,\bp)$ by restricting it to the joint set.
Such an infinitesimal motion is called {\em trivial}. 
$(G,{\bm p})$ is said to be {\em infinitesimally rigid} if every motion is trivial.
An infinitesimally rigid framework $(G,{\bm p})$ is {\em minimally infinitesimally rigid} if the framework is not infinitesimally rigid after removing any edge.

Instead of this familiar notation, we shall introduce a different (but, of course, equivalent) definition of the infinitesimal rigidity of 2-dimensional bar-joint frameworks, 
used in \cite{tay1991linking,tay:89,tanigawa2010}.
Namely, we shall define bar-joint frameworks in terms of the body-bar model, where  
each bar-joint framework is considered as a special case of body-bar frameworks by regarding each joint as a ``0-dimensional'' body.
Notice that $\bp$ determines a mapping  
${\bm b}:E\rightarrow Gr(2,\mathbb{R}^{3})$ by ${\bm b}_{uv}=(\bp(u),1)\wedge (\bp(v),1)$ for $uv\in E$.
Then $(G,\bp)$ is equivalent to the 2-dimensional body-bar framework $(G,{\bm b})$, which satisfies a special incidence condition between $\bp$ and ${\bm b}$:
\begin{equation}
\label{eq:bar_joint_incidence}
\langle (\bp(v),1), {\bm b}_e\rangle =0 \qquad \text{if $e\in E$ is incident to $v\in V$}.
\end{equation}
$(G,{\bm b})$ is said to be the body-bar framework {\em derived from} the bar-joint framework $(G,\bp)$.

Recall that an infinitesimal motion of a 2-dimensional body-bar framework $(G,{\bm b})$ is a mapping $\bmm:V\rightarrow \bigwedge^2\mathbb{R}^3$ satisfying bar-constraints (\ref{eq:body_bar_bar_const}).
$\bmm$ is always a motion of $(G,{\bm b})$ if $\bmm(u)=\bmm(v)$ for $u,v\in V$, and such motions are called  trivial.
If $(G,{\bm b})$ satisfies incidence condition (\ref{eq:bar_joint_incidence}), $(G,{\bm b})$ always has additional $|V|$ independent motions:
for each $v\in V$ define $\bmm_v$ by $\bmm_v(u)=0$ for $u\in V\setminus \{v\}$ and $\bmm_v(v)=(\bp(v),1)$; then $\bmm_v$ satisfies (\ref{eq:body_bar_bar_const}) by (\ref{eq:bar_joint_incidence}).
Such ${\bm m}_v$ is called a {\em trivial dangling (around $v$)}. 
A body-bar framework $(G,{\bm b})$ is said to be {\em bar-joint-rigid} if every infinitesimal motion is a linear combination of trivial motions and trivial danglings.

\begin{proposition}
\label{prop:equivalence}
A 2-dimensional bar-joint framework $(G,\bp)$ is infinitesimally rigid if and only if  the body-bar framework $(G,{\bm b})$ derived from $(G,\bp)$ is bar-joint-rigid.
\end{proposition}
\begin{proof}
This is immediate from the fact that any infinitesimal motion of a body can be described as a linear combination of an infinitesimal rotation around a point $p$ in the body and translations of $p$.
\end{proof}
  
Notice that in the above discussion we only require homogeneous coordinates of joints when constructing the derived body-bar frameworks.
We can thus naturally extend the notion of bar-joint frameworks to the projective plane,
whose rigidity is defined in terms of the derived body-bar frameworks. 
We also remark that bar-joint frameworks in the real projective space can be equivalently defined in terms of static rigidity, see, e.g.,~\cite{crapo:1982,whiteley:83:cone}.

\subsection{Bar-joint Frameworks with Bar-boundary}
\label{subsec:bar_joint_bar}
A {\em 2-dimensional bar-joint framework with bar-boundary} is a tuple $(G,\bR;{\bm p},{\bm b}^{\circ})$ such that
\begin{itemize}
\item $G=(V,E)$ is a graph and $\bR$ is a multiset of vertices;
\item $\bp:V\rightarrow \mathbb{R}^2$ is a joint-configuration;
\item ${\bm b}^{\circ}:\bR\rightarrow Gr(2,\mathbb{R}^3)$ is a configuration of bar-boundary, which must satisfy incidence condition,
\begin{equation*}
\langle ({\bm p}(v),1), {\bm b}^{\circ}_r\rangle =0 \qquad \text{if } r\in \bR_v.
\end{equation*}
\end{itemize}
Namely, as in the body-bar case,  we have introduced abstract signs $\bR$ of bar-boundary and their realization ${\bm b}^{\circ}$,
where ${\bm b}^{\circ}_r$ denotes the Pl{\"u}cker coordinate of a bar connecting joint ${\bm p}(v)$ and the external environment for each $r\in \bR_v$.

Following the conventional definition, an infinitesimal motion of $(G,\bR;{\bm p},{\bm b}^{\circ})$ is defined as a mapping $\bmm:V\rightarrow \mathbb{R}^2$ 
satisfying bar-constraints (\ref{eq:bar_joint_motion1}) as well as bar-boundary constraints: for each $v\in V$,
\begin{equation*}
\langle (\bmm(v),0), {\bm b}^{\circ}_r\rangle =0 \qquad \text{for } r\in \bR_v,
\end{equation*}
and $(G,\bR;{\bm p},{\bm b}^{\circ})$ is said to be infinitesimally rigid if there is no nonzero motion.

Let us rewrite this notion in terms of the body-bar model, again.
As in the previous subsection, ${\bp}$ determines the bar-configuration ${\bm b}:E\rightarrow Gr(2,\mathbb{R}^3)$ by ${\bm b}_{uv}=(p(u),1)\wedge (p(v),1)$,
and thus {\em the body-bar framework with bar-boundary} $(G,\bR;{\bm b},{\bm b}^{\circ})$ is  {\em derived from} $(G,\bR;\bp,{\bm b}^{\circ})$.
In general, a body-bar framework with bar-boundary $(G,\bR;{\bm b},{\bm b}^{\circ})$ is said to be {\em bar-joint-rigid} if every possible motion is a linear combination of trivial danglings.
As in the previous subsection we have the following:
\begin{proposition}
\label{prop:equivalence2}
A 2-dimensional bar-joint framework $(G,\bR;\bp,{\bm b}^{\circ})$ with bar-boundary  is infinitesimally rigid 
if and only if  the body-bar framework with bar-boundary derived from $(G,\bR;\bp,{\bm b}^{\circ})$ is bar-joint-rigid.
\end{proposition}

We now extend the notion of bar-joint frameworks to the real projective plane.
A 2-dimensional bar-joint framework with bar-boundary is  defined in the projective plane by a tuple $(G,\bR;\tilde{\bm p},{\bm b}^{\circ})$, where
\begin{itemize}
\item $G=(V,E)$ is a graph and $\bR$ is a multiset of vertices;
\item $\tilde{\bp}:v\in V\mapsto [\tilde{\bp}_v]\in \mathbb{P}^2$ is an injective mapping;
\item ${\bm b}^{\circ}:\bR\rightarrow Gr(2,\mathbb{R}^3)$ is a configuration of bar-boundary satisfying 
incidence condition between  joints and boundary-bars:
\begin{equation}
\label{eq:bar_joint_boundary_incidence}
\langle \tilde{\bp}_v, {\bm b}^{\circ}_r\rangle =0 \qquad \text{if } r\in \bR_v.
\end{equation}
\end{itemize}
Since $\tilde{\bp}$ determines a  bar-configuration ${\bm b}:E\rightarrow Gr(2,\mathbb{R}^3)$, 
$(G,\bR;\tilde{\bp},{\bm b}^{\circ})$ derives a body-bar framework with bar-boundary $(G,\bR;{\bm b},{\bm b}^{\circ})$ 
with incidence property between $\bp$ and ${\bm b}$:
\begin{equation}
\label{eq:bar_joint_incidence2}
\langle \tilde{\bp}_v, {\bm b}_e\rangle =0 \qquad \text{if $e\in E$ is incident to $v\in V$}.
\end{equation}
$(G,\bR;\tilde{\bp},{\bm b}^{\circ})$ is said to be infinitesimally rigid if the derived framework $(G,\bR;{\bm b},{\bm b}^{\circ})$ is bar-joint rigid.

We now provide an extension of Laman's theorem. 
The proof is again based on Theorem~\ref{thm:partition}, and its idea is essentially from Tay~\cite{tay:93}.
Unfortunately, in this case (compared with body-bar case), we need an assumption of ``generality'' of bar-boundary configurations:
A finite set of projective lines is in {\em general position} if no three lines of the set intersects at a point. 
\begin{theorem}
\label{thm:bar_joint_bar}
Let $G=(V,E)$ be a graph, $\bR$ be the multiset of vertices, and ${\bm b}^{\circ}:\bR\rightarrow Gr(2,\mathbb{R}^3)$.
Suppose ${\bm b}^{\circ}(\bR)$ is in general position.
Then there exists $\tilde{\bm p}:V\rightarrow \mathbb{P}^2$ such that
$(G,\bR;\tilde{\bm p},{\bm b}^{\circ})$ is a minimally infinitesimally rigid bar-joint framework with bar-boundary if and only if
\begin{itemize}
\item $|\bR_v|\leq 2$ and $\{{\bm b}^{\circ}_r\mid r\in \bR_v \}$ is linearly independent for $v\in V$,
\item $|F|+|\bR_F|\leq 2|V(F)|-3+\dim(\{{\bm b}^{\circ}_r\mid r\in \bR_F\})$ for any non-empty $F\subseteq E$,
\item $|E|+|\bR|=2|V|$.
\end{itemize} 
\end{theorem}
\begin{proof}
We only prove the sufficiency.
(The necessity can be shown in an identical manner to Theorem~\ref{thm:body_bar_bar}).
Define a linear matroid ${\cal M}$ on $\bR$ represented by ${\bm b}^{\circ}_r \ (r\in \bR)$, 
and let $k$ be the rank of ${\cal M}$.
As in the proof of Theorem~\ref{thm:body_bar_bar}, we  have $k=3$ from the counting condition.
We first construct a special rooted-tree decomposition based on Theorem~\ref{thm:partition}.
\begin{claim}
\label{claim:dec1}
$G$ admits a rooted-tree decomposition $\{(T_1, r_1), \dots, (T_t,r_t)\}$ such that 
\begin{itemize}
\item[(i)] it is a basic decomposition with respect to the truncation ${\cal M}^{\downarrow}$ of ${\cal M}$, and
\item[(ii)] for any $X\subseteq V$ with $|X|\geq 2$ at most one set among $\{T_i\cap E[X]\mid 1\leq i\leq t \}$ forms a spanning tree on $X$,
where $E[X]=\{uv\in E\mid u,v\in X\}$. 
\end{itemize}
\end{claim}
\begin{proof}
Since $r_{\calM}(R_F)\leq r_{\calM^{\downarrow}}(R_F)+1$, we have $|F|+|\bR_F|\leq 2|V(F)|-2+r_{\calM^{\downarrow}}(\bR_F)$ for any non-empty $F\subseteq E$.
Hence  $E$ admits a basic rooted-tree decomposition $\{(T_1,r_1),\dots, (T_t,r_t)\}$ with respect to ${\cal M}^{\downarrow}$ by Theorem~\ref{thm:partition}.
Suppose this decomposition does not satisfy (ii). Then there is an $X\subseteq V$ with $|X|\geq 2$ such that
at least two sets among $\{T_i\cap E[X]\mid 1\leq i\leq t \}$, say $T_1\cap E[X]$ and $T_2\cap E[X]$, form spanning trees on $X$.
Let $F$ be the union.
Since the decomposition is basic with respect to ${\cal M}^{\downarrow}$,
every vertex in $V(F)$ is spanned by only $(T_1,r_1)$ and $(T_2,r_2)$.
Thus, $|\bR_F|=\dim\{{\bm b}^{\circ}_r\mid r\in \bR_F\}$.
From $|F|=2|V(F)|-2$, we obtain $|F|+|\bR_F|>2|V(F)|-3+\dim\{{\bm b}^{\circ}_r\mid r\in \bR_F\}$, a contradiction.
\end{proof}

Take a rooted-tree decomposition $\{(T_1,r_1), \dots, (T_t,r_t)\}$ shown in Claim~\ref{claim:dec1}. 
We define ${\bm b}:E\rightarrow Gr(2,\mathbb{R}^3)$ by
\[
{\bm b}(e)={\bm b}^{\circ}(r_i) \qquad \text{ if } e\in T_i.
\]
Since the decomposition is basic with respect to ${\cal M}^{\downarrow}$, each vertex is spanned by exactly two $(T_i,r_i)$ and $(T_j,r_j)$ among them,
such that ${\bm b}(r_i)\neq {\bm b}(r_j)$.
We can thus define $\tilde{\bp}:V\rightarrow \mathbb{P}^2$ by 
\begin{equation*} 
\tilde{\bp}(v)={\bm b}^{\circ}(r_{i})\cap {\bm b}^{\circ}(r_j) \qquad \text{if $v$ is spanned by $(T_i,r_i)$ and $(T_j,r_j)$}.
\end{equation*}
Clearly, $\tilde{\bp}$, ${\bm b}$ and ${\bm b}^{\circ}$ satisfy the incidence conditions  (\ref{eq:bar_joint_boundary_incidence}) and (\ref{eq:bar_joint_incidence2}).

As in the proof of Theorem~\ref{thm:body_bar_bar}, it can be easily checked that 
the possible infinitesimal motions of the body-bar framework $(G,\bR;{\bm b},{\bm b}^{\circ})$ are linear combinations of  trivial danglings;
therefore, $(G,\bR;{\bm b},{\bm b}^{\circ})$ is bar-joint-rigid.
However, $\tilde{\bp}$ may not be injective, which means that $(G,\bR;\tilde{\bp},{\bm b}^{\circ})$ may not be a bar-joint framework.
We now show that $\tilde{\bp}$ can be continuously perturbed so that $\tilde{\bp}$ is injective keeping the bar-joint-rigidity.

Since ${\bm b}^{\circ}$ is in general position, $\tilde{\bp}(u)=\tilde{\bp}(v)$ holds if and only if
$u$ and $v$ are spanned by the same two rooted-trees in the decomposition.
Suppose there exists a set $X$ of vertices with $|X|\geq 2$ which are spanned by the same two rooted-trees, say 
$(T_1,r_1)$ and $(T_2,r_2)$.
By (ii) of Claim~\ref{claim:dec1} we may assume that $T_2\cap E[X]$ is not  a spanning tree on $X$.
Since $T_2\cap E[X]$ does not span all elements of $X$, we can take a proper subset $X'$ of $X$ such that $r_2\notin X'$ and every edge connecting between 
$X'$ and $X\setminus X'$ belongs to $T_1$. 
To resolve the point-coincidence between $\tilde{\bp}(X')$ and $\tilde{\bp}(X\setminus X')$,
we continuously move $\tilde{\bp}(X')$  along the line ${\bm b}^{\circ}(r_1)$ keeping the coincidence inside ${\bp}(X')$.
The lines $\{{\bm b}(uv)\mid u,\in X', v\in V\setminus X\}$ are simultaneously moved to keep the incidence (\ref{eq:bar_joint_incidence2}), whose directions are continously changed.   
If the displacement is small enough, the dimension of the motion space does not change since all coordinates are continuously changed.
Also, since $r_i\notin X'$ for any $r_i\in \bR\setminus \{r_1\}$, the incidence (\ref{eq:bar_joint_boundary_incidence}) is preserved.
Applying this procedure repeatedly, $\tilde{\bp}$ can be converted to an injective mapping keeping the bar-joint-rigidity, 
and we obtain an infinitesimally rigid  bar-joint framework $(G,\bR;\tilde{\bp},{\bm b}^{\circ})$.
\end{proof}

\noindent
{\bf Remark.}
The special decomposition presented in the proof of Theorem~\ref{thm:bar_joint_bar} is an analog of a so-called {\em proper 3tree2 decomposition}, 
introduced by Crapo~\cite{crapo:1990} for an alternative characterization of 2-dimensional generic rigidity.

\medskip

\noindent
{\bf Remark.}
The statement of Theorem~\ref{thm:bar_joint_bar} can be converted to a purely combinatorial form due to the simplicity of the lattice of the linear matroid ${\cal M}$ represented by ${\bm b}^{\circ}$.
Let us assign a color to each element in $\bR$ such that $r_i$ and $r_j$ have the different colors if and only if ${\bm b}^{\circ}(r_i)\neq {\bm b}^{\circ}(r_j)$.
A matroid ${\cal N}=(\bR,{\cal J})$ can be defined such that $\bR'\subseteq \bR$ is independent if and only if all elements of $\bR'$ have distinct colors and $|\bR'|\leq 3$. 
Then ${\cal N}$ is isomorphic to ${\cal M}$ if ${\bm b}^{\circ}(\bR)$ is in general position.
This implies that the counting condition of Theorem~\ref{thm:bar_joint_bar} can be written in terms of ${\cal N}$ as follows:
\begin{itemize}
\item For each $v\in V$, $|R_v|\leq 2$, and $r$ and $r'$ have distinct colors if $\bR_v=\{r,r'\}$;
\item $|F|+|\bR_F|\leq 2|V(F)|-3+\min\{3,c(\bR_F)\}$ for any non-empty $F\subseteq E$, where $c(\bR_F)$ denotes the number of colors in $\bR_F$;
\item $|E|+|\bR|=2|V|$.
\end{itemize} 
In \cite{rooted_forest}, we showed how to check the counting condition of this type in $O(|V|^2)$ time.

\subsection{Pinned Bar-joint Frameworks}
A $2$-dimensional {\em pinned bar-joint framework} is defined as $(G,X, {\bm p})$, where
\begin{itemize}
\item $G=(V,E)$ is a graph;
\item $X$ is a subset of $V$;
\item ${\bm p}:V\rightarrow \mathbb{R}^{2}$ is a joint-configuration.
\end{itemize}
An infinitesimal motion of $(G,X,\bp)$ is a mapping ${\bm m}:V\rightarrow \mathbb{R}^2$ satisfying bar-constraints (\ref{eq:bar_joint_motion1})
as well as additional pin-constraints; for each $v\in V$
\begin{equation*}
{\bm m}(v)=0 \qquad \text{ for } v\in X.
\end{equation*}
$(G,X,{\bm p})$ is said to be {\em infinitesimally rigid} if there is no nonzero motion.

As before, we can define a pinned bar-joint framework in the real projective plane by $(G,X,\tilde{\bm p})$, where 
$G=(V,E)$ is a graph, $X\subseteq V$, and $\tilde{\bm p}:V\rightarrow \mathbb{P}^2$.

In two dimensional case, the pinning down a point is equivalent to connecting that point with external environment by two any distinct bars.
We can thus consider pinned bar-joint frameworks as a special case of bar-joint frameworks with bar-boundary, where configurations of bar-boundary can be in general position.
It is thus straightforward to see the following characterization of pinned bar-joint frameworks in the real projective plane from Theorem~\ref{thm:bar_joint_bar}.
\begin{theorem}[\cite{servatius2010combinatorial}]
\label{thm:pin_laman_projective}
Let $G=(V,E)$ be a graph, $X$ be a vertex subset, and $\tilde{\bm p}_X:X\rightarrow \mathbb{P}^2$ be an injective mapping.
Define $f_X:2^E\rightarrow \{0,2,3\}$ by
\begin{equation}
\label{eq:f_X}
f_{X}(F)=\begin{cases}
0 & \text{ if } |X\cap V(F)|=0 \\ 
2 & \text{ if } |X\cap V(F)|=1 \\
3 & \text{ otherwise. }
\end{cases}
\end{equation}
Then, there is a joint configuration $\tilde{\bm p}:V\rightarrow \mathbb{P}^2$ extending $\tilde{\bm p}_X$
such that the pinned bar-joint framework $(G,X,\tilde{\bm p})$ is minimally infinitesimally rigid if and only if
\begin{itemize}
\item $|E|=2|V\setminus X|$,
\item $|F|\leq 2|V(F)\setminus X|-3+f_X(F)$ for any non-empty $F\subseteq E$.
\end{itemize}
\end{theorem}

The combinatorial condition of Theorem~\ref{thm:pin_laman_projective} actually characterizes not only infinitesimal rigidity but also rigidity in the 2-dimensional Euclidean space.
Let us identify $\bp:V\rightarrow \mathbb{R}^2$ with a point $\bp\in \mathbb{R}^{2|V|}$ 
and consider $f_{G}:\mathbb{R}^{2|V|}\rightarrow \mathbb{R}^{E}$ defined by $f_{G}(\bp)=(\dots,(\bp(u)-\bp(v))^2,\dots)\in\mathbb{R}^E$.
Substituting $\bp(X)=\tilde{\bp}(X)$,
$f_G$ is reduced to a mapping $f_{G,X}:\mathbb{R}^{2|V\setminus X|}\rightarrow \mathbb{R}^{|E|}$ of the remaining parameter $\bp(V\setminus X)\in \mathbb{R}^{2|V\setminus X|}$.
Then $(G,X,\bp)$ is said to be {\em rigid} if $f^{-1}_{G,X}(f_{G,X}(\bp(V\setminus X))$ is an isolated point in $\mathbb{R}^{2|V\setminus X|}$.

Let $R(G,X,\bp)$ be the Jacobian of $f_{G,X}$ at $\bp(V\setminus X)$.
It is easy to observe that $(G,\bp)$ is infinitesimally rigid if and only if the rank of $R(G,X,\bp)$ is equal to $2|V\setminus X|$. 
$\bp$ is called {\em regular} if the rank of $R(G,X,\bp)$ is maximized over all joint-configurations with $\bp(X)=\tilde{\bp}(X)$. 
Notice that $\{\bp(V\setminus X)\in \mathbb{R}^{2|V\setminus X|}\mid \bp\in \mathbb{R}^{2|V|} \text{ is regular}\}$ forms a dense open subset of $\mathbb{R}^{2|V\setminus X|}$.
Moreover, applying the same argument as Asimow and Roth~\cite{asimow1978}, we see that $(G,X,\bp)$ is rigid if and only if $(G,X,\bp)$ is infinitesimally rigid if $\bp$ is regular.
Consequently, for almost all joint-configurations extending $\tilde{\bp}$, the rigidity of pinned bar-joint frameworks is characterized by the counting condition given in Theorem~\ref{thm:pin_laman_projective}.

\subsection{Bar-joint-slider Frameworks}
\label{subsec:bar_joint_slider}
A {\em 2-dimensional bar-joint-slider framework} is a bar-joint framework some of whose joints are constrained by sliders as shown in Figure~\ref{fig:bar_slider}(a).
Such a slider restricts the possible motions of a joint to the direction of the slider. 
We thus define a 2-dimensional bar-joint-slider framework by $(G,\bR;\bp,\bd)$, where
\begin{itemize}
\item $G=(V,E)$ be a graph and $\bR$ is a multiset of vertices;
\item $\bp:V \rightarrow \mathbb{R}^2$ is a joint-configuration;
\item $\bd:r\in \bR \mapsto [\bd_r]\in \mathbb{P}^1$ is a realization of sliders.
\end{itemize}
In this setting $\bd_r$ indicates the direction of the slider corresponding to $r\in \bR$. 

An {\em infinitesimal motion} is defined as a mapping $\bmm:V\rightarrow \mathbb{R}^2$ satisfying not only bar-constraints (\ref{eq:bar_joint_motion1})
but also slider constraints:
\begin{equation}
\langle \bmm(v),\bd_r^{\bot}\rangle=0 \qquad \text{ for each $r\in \bR_v$} \label{eq:bar_slider_external}
\end{equation}
where $\bd_r^{\bot}$ denotes the unit vector orthogonal to $\bd_r\in \mathbb{R}^2$.
$(G,\bR;\bp,\bd)$ is said to be infinitesimally rigid if there is no nonzero motion.

The following characterization of minimal rigidity was given by us in \cite{isaac2009}, which generalizes a result of \cite{streinu2010slider} to non-generic case.
The theorem is now immediate from Theorem~\ref{thm:pin_laman_projective}.
\begin{theorem}
\label{thm:bar_slider}
Let $G$ be a graph, $\bR$ be a multiset of vertices,  and $\bd:\bR\rightarrow \mathbb{P}^1$ be a mapping from $r\in \bR$ to the direction of the slider corresponding to $r$.
Then, there exits
a joint-configuration $\bp$  such that
the 2-dimensional bar-joint-slider framework $(G,\bR;\bp,\bd)$ is minimally infinitesimally rigid if and only if
\begin{itemize}
\item $|F|+|\bR_F|\leq 2|V(F)|-2+\min\{2,c(\bR_F)\}$ for any non-empty $F\subseteq E$, 
\item $|F|\leq 2|V(F)|-3$ for any non-empty $F\subseteq E$,
\item $|E|+|\bR|=2|V|$,
\end{itemize}
where $c(\bR_F)$ denotes the number of distinct directions among $\bd(\bR_F)$.
\end{theorem}
\begin{proof}
In the analysis of infinitesimal rigidity, a slider constraint is equivalent to a bar-constraint between  the external environment and the corresponding joint, 
with the bar orthogonal to the direction of the slider.
Such external bar-constraints intersects at a point at infinity if the corresponding sliders have the same direction.
This means that a 2-dimensional bar-joint-slider framework can be converted to a pinned bar-joint framework in the real projective plane with the same rigidity property.
See Figure~\ref{fig:bar_slider}(b) for an example.

Let us see this conversion in more detail.
We shall consider an auxiliary graph $G'=(V',E')$ whose vertex set $V'$ is $V\cup \bR$ (as a multiset), and
$u$ and $v$ are linked by an edge if and only if (i) $u, v\in V$ and $uv\in E$ or 
(ii) $v\in V$ and  $u\in \bR_v$.
Define $\tilde{\bp}:V'\rightarrow \mathbb{P}^2$ by $\tilde{\bp}(v)=[\bp(v),1]$ for $v\in V$ and $\tilde{\bp}(r)=[\bd_r^{\bot},0]$ for $r\in \bR$.
Also, let $X=\bR$.   
Then $(G',X,\tilde{p})$ is a pinned bar-joint framework in the real projective plane, which is infinitesimally rigid if and only if 
the original bar-joint-slider framework $(G,\bR;\bp,\bd)$ is infinitesimally rigid.

It is routine to check the equivalence of two counting conditions: the one given in the statement for $G$ and the one of Theorem~\ref{thm:pin_laman_projective} for $G'$.
\end{proof}

\noindent
{\bf Remark.} 
Let $\cal M$ be the linear matroid on $\bR$ represented by $\bd_r$.
An edge set satisfying the first and the second condition of Theorem~\ref{thm:bar_slider} is a common independent set of ${\cal N}(f_{\calM,2})$ and 
the generic 2-rigidity matroid, which is the matroid induced by $2|V(F)|-3$.
However, since the function $\hat{f}_{\calM,2}$ defined by
\begin{equation*}
\hat{f}_{\calM,2}(F)=\min\{f_{\calM,2}(F), 2|V(F)|-3\}
\end{equation*}
happens to be submodular,
Theorem~\ref{thm:bar_slider} is indeed a characterization in terms of a matroid.

%

\begin{figure}[t]
\centering
\begin{minipage}{0.4\textwidth}
\centering
\includegraphics[scale=1]{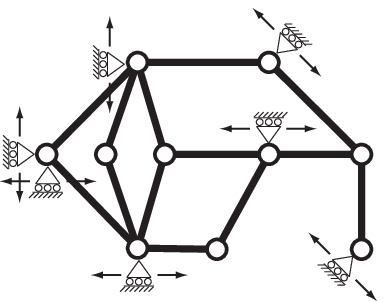}
\par
(a)
\end{minipage}
\begin{minipage}{0.4\textwidth}
\centering
\includegraphics[scale=0.8]{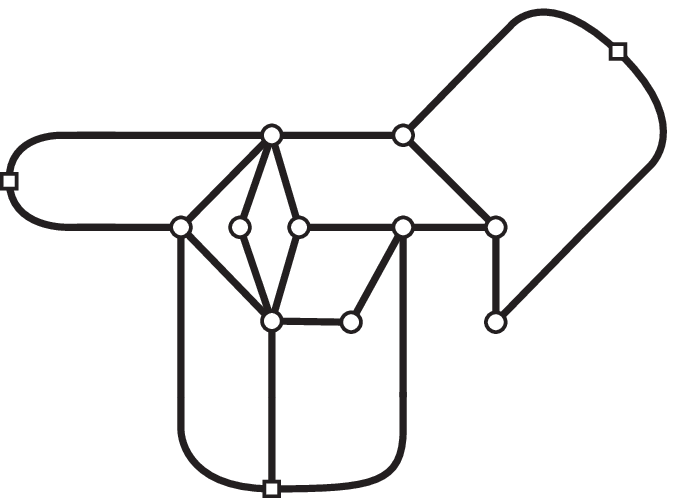}
\par
(b)
\end{minipage}
\caption{(a) A bar-slider framework and (b) the corresponding pinned bar-joint framework, where the squared vertices are pinned at infinity.}
\label{fig:bar_slider}
\end{figure}

\section{Concluding Remarks}
We have presented extensions of Nash-Williams tree-partition theorem and Tutte-Nash-Williams tree-packing theorem, by relaxing the ``spanning'' condition
to a matroid condition.

An interesting open problem is to develop a further extension of Theorem~\ref{thm:partition}.
Let us consider the following natural extension.
Suppose we are given a graph $G$ with roots $\bR$, a matroid ${\cal M}$ on $\bR$ of rank $k$, and $d:V\rightarrow \{1,\dots, k\}$.
Can we decide whether $(G,\bR)$ contains edge-disjoint rooted-trees $(T_1,r_1),\dots, (T_t,r_t)$ 
such that $\{r_i\in \bR \mid v\in V(T_i,r_i)\}$ is an independent set of size $d(v)$ for each $v\in V$?
This problem is however shown to be NP-hard even if $|\bR|=2$ and ${\cal M}$ is free,
by the reduction from the problem of deciding the decomposability of a hypergraph into two connected spanning sub-hypergraphs, which is known to be NP-complete~\cite{frank:2003}.

The underlying combinatorial structure of basic decompositions, especially a relation to matroid union, is currently unclear. 
As remarked in introduction, if ${\cal M}$ can be written as the direct sum of $k$ matroids of rank $1$, Theorem~\ref{thm:partition} straightforwardly follows from matroid union theorem.
%

Recall that Theorem~\ref{thm:bar_joint_bar} characterizes the infinitesimal rigidity of bar-joint frameworks with bar-boundary in {\em general} positions. 
We leave it as an open problem whether the assumption of generality can be dropped.

\section*{Acknowledgment}
We thank Satoru Iwata for pointing out the extendability of Theorem~\ref{thm:partition} from linear matroids to general matroids.
The description of body-bar frameworks given in Appendix~\ref{app:bar_description} is based on a discussion with Ileana Streinu and Ciprian Borcea.
The first author is supported by JSPS Grant-in-Aid for Scientific Research (B).

\bibliography{tani20110702}{}

\appendix
\section{Description of Bar-constraints}
\label{app:bar_description}
We may coordinatize the exterior product $\mathbb{R}^{d}\wedge \mathbb{R}^d$ as follows:
For $a=(a_1,a_2,\dots, a_d)\in \mathbb{R}^d$ and $b=(b_1,b_2,\dots, b_d)\in \mathbb{R}^d$, 
\begin{equation}
a\wedge b= 
\Bigg(
\Bvector{(1,2)}{\begin{vmatrix} a_1 & a_2 \\ b_1 & b_2 \end{vmatrix}}, 
\Bvector{(1,3)}{-\begin{vmatrix}a_1 & a_3 \\ b_1 & b_3 \end{vmatrix}}, 
\Bvector{}{\ \ \cdots \ \ },  
\Bvector{(i,j)}{(-1)^{i+j+1}\begin{vmatrix}a_i & a_j \\ b_i & b_j \end{vmatrix}}, 
\Bvector{}{\ \ \cdots \ \ }, 
\Bvector{(d-1,d)}{\begin{vmatrix}a_{d-1} & a_d \\ b_{d-1} & b_d \end{vmatrix}} 
\Bigg)\in \mathbb{R}^{{d \choose 2}}.
\end{equation}

\vspace{\baselineskip}

Suppose we are given
rigid bodies $B_1$ and $B_2$ in $\mathbb{R}^d$, which can be identified
with a pair $(p_i,M_i)$ of a point $p_i\in \mathbb{R}^d$ and an orthogonal matrix $M_i\in SO(d)$ for each $i=1,2$.
Namely, each $(p_i,M_i)$ is a local Cartesian coordinate system for each body.
We consider a situation, where the bodies $B_1$ and $B_2$ are connected by a bar.
We denote the endpoints of the bars by $p_1+M_1q_1$ and $p_2+M_2q_2$, where $q_i$ is the coordinate of each endpoint (joint) in the coordinate system 
of each body.

The constraint by the bar can be written by
\begin{equation}
\langle p_2+M_2q_2-p_1-M_1q_1,
p_2+M_2q_2-p_1-M_1q_1\rangle
=\ell^2
\end{equation}
for some $\ell\in \mathbb{R}$.
If we take the differentiation with variables $p_i$ and $M_i$, we get 
\begin{equation}
\langle p_2+M_2q_2-p_1-M_1q_1,
\dot{p}_2+{\dot M}_2q_2-\dot{p}_1-\dot{M}_1q_1 \rangle   
=0
\end{equation}
We may simply assume $p_i=0$ and $M_i=I_d$. Then by setting $h=q_2-q_1$ and $\dot{M}_i=A_i$ with a skew-symmetric matrix $A_i$,
\begin{equation}
\label{eq:const}
\langle h , \dot{p}_2+A_2q_2-\dot{p}_1-A_1q_1\rangle =0.
\end{equation}

Also we denote a skew-symmetric matrix $A$ by  
\begin{equation}
A=\begin{pmatrix} 
0 & -w_{1,2} & \cdots & \cdots & \cdots & \cdots & (-1)^{d+1}w_{1,d} \\ 
w_{1,2} & 0 &  &   &       &        & \vdots \\
 &    &         &   &       &        & \\
\vdots &    &  \ddots       &         & (-1)^{i+j}w_{i,j} &  & \vdots \\
 &    &         &   &       &        & \\
\vdots &    &   &     0    &       &        & \vdots \\
 &    &         &   &       &        & \\
\vdots &    & (-1)^{i+j+1}w_{i,j} &   &   \ddots     &    &  \vdots \\ 
 &    &         &   &       &        & \\
\vdots       &    &         &         &       &     0   & w_{d-1,d} \\
(-1)^{d}w_{1,d}       & \cdots    &   \cdots      &  \cdots       &  \cdots     &  -w_{d-1,d}      & 0   \\
\end{pmatrix}
\end{equation}
and let  $w=\begin{pmatrix}w_{1,2} & w_{1,3} & \cdots & w_{d-1,d}\end{pmatrix}\in \mathbb{R}^{{d \choose 2}}$.
Then, for any $h\in \mathbb{R}^d$ and $q\in \mathbb{R}^d$, we have
\begin{equation}
\label{eq:1}
\langle h, Aq \rangle=\langle q\wedge h, w \rangle.
\end{equation}
Therefore, we can simply describe the infinitesimal bar-constraint (\ref{eq:const}) by 
\begin{equation}
\label{eq:const2}
\langle q_2-q_1, \dot{p}_2-\dot{p}_1\rangle + \langle q_2\wedge q_1, w_2-w_1 \rangle =0,
\end{equation}
where $w_1\in \mathbb{R}^{{d\choose 2}}$ and $w_2\in \mathbb{R}^{{d\choose 2}}$ denote the ${d\choose 2}$-dimensional vectors corresponding to $A_1$ and $A_2$, respectively. 

We call a pair $s_i=(w_i,p_i)\in \mathbb{R}^{d\choose 2}\times \mathbb{R}^d$ a {\em screw motion}, which can be identified with a vector in $\mathbb{R}^{d+1\choose 2}$. 
Using the homogeneous coordinate of $q_i$ in $\mathbb{P}^d$, (\ref{eq:const2}) is written as 
\begin{equation}
\label{eq:const4}
\langle (q_2,1)\wedge (q_1,1)), s_2-s_1\rangle =0.
\end{equation}

\end{document}